\documentclass[a4paper,11pt]{amsart}
\usepackage{graphics}
\usepackage{amscd}
\usepackage{comment}
\usepackage{oldgerm}
\usepackage{amsmath}
\usepackage{amssymb}
\usepackage{amsthm}
\newtheorem{thm}{Theorem}[section]
\newtheorem{prop}[thm]{Proposition}
\newtheorem{lem}[thm]{Lemma}
\newtheorem{cor}[thm]{Corollary}
\newtheorem{definition}[thm]{Definition}
\newtheorem{rem}[thm]{Remark}

\newcommand{\Z}{\mathbb{Z}}
\newcommand{\N}{\mathbb{N}}
\newcommand{\R}{\mathbb{R}}
\newcommand{\C}{\mathbb{C}}

\newcommand{\geh}{\mathfrak{g}}

\newcommand{\GH}{\stackrel{{\rm GH}}{\longrightarrow}}
\newcommand{\mGH}{\stackrel{{\rm mGH}}{\longrightarrow}}

\newcommand{\GmGH}{\stackrel{G\mathchar`-{\rm mGH}}{\longrightarrow}}
\newcommand{\Mosco}{\stackrel{{\rm M}}{\to}}
\newcommand{\str}{\stackrel{{\rm str}}{\to}}
\newcommand{\cpt}{\stackrel{{\rm cpt}}{\to}}
\begin{document}
\title[On spectral convergence of vector bundles]{On spectral convergence of vector bundles and 
convergence of principal bundles}
\author{Kota Hattori}
\address{Keio University, 
3-14-1 Hiyoshi, Kohoku, Yokohama 223-8522, Japan}
\email{hattori@math.keio.ac.jp}
\thanks{Partially 
supported by Grant-in-Aid for Young Scientists (B) Grant Number
16K17598 and by JSPS 
Core-to-Core Program, ``Foundation of a Global Research Cooperative Center in Mathematics
focused on Number Theory and Geometry''.}
\maketitle
\begin{abstract}
In this article we consider the continuity of 
the eigenvalues of the connection Laplacian of 
connections on vector bundles over Riemannian manifolds. 
To show it, we consider 
the equivariant version of the measured Gromov-Hausdorff 
topology on the space of metric measure spaces 
with isometric actions, and apply it to 
the total spaces of principal bundles equipped 
with connections over Riemannian manifolds.
\end{abstract}
\section{Introduction}
For a closed connected 
Riemannian manifold $(X,g)$, 
the Laplace operator is defined by 
\begin{align*}
\Delta f:= d^*d f
\end{align*}
for any smooth functions $f$ on $X$, 
where $d^*$ is the formal adjoint of 
exterior derivative $d$. 
As one of the fundamental results 
of harmonic analysis, we obtain the 
orthonormal basis $\{ f_j\}_j$ of $L^2(X)$ 
and eigenvalues 
\begin{align*}
0=\lambda_1<\lambda_2\le\lambda_3\le\cdots
\end{align*}
such that $\lim_{j\to\infty}\lambda_j = \infty$ 
and $\Delta f_j=\lambda_j f_j$. 

Here, we consider the function $\lambda_j$, 
which maps the isometric class of 
closed Riemannian manifolds to the 
$j$-th eigenvalue of its Laplacian. 
In \cite{Fukaya1987}, Fukaya has shown the continuity of 
$\lambda_j$ with respect to the measured 
Gromov-Hausdorff 
topology under the assumptions that 
the sectional curvatures of 
Riemannian manifolds are bounded. 
He also conjectured that the continuity should 
holds if the lower bound of the 
Ricci curvatures are given. 

His conjecture was solved by 
Cheeger and Colding in \cite{Cheeger-Colding3}. 
They defined the Laplacian on the 
metric measure 
space $(X,d,\mu)$ which is the 
measured 
Gromov-Hausdorff limit of a sequence of 
smooth connected closed Riemannian manifolds 
$\{ (X_i,g_i)\}$ 
with 
\begin{align*}
{\rm diam}(X_i) \le D,\quad
{\rm Ric}_{g_i}\ge \kappa g_i
\end{align*}
for some constant $D>0$ and $\kappa\in\R$, 
then they showed that  
\begin{align*}
\lim_{i\to \infty}\lambda_j{(X_i,g_i)} = \lambda_j{(X,d,\mu)}.
\end{align*}
Here, $(X_i,g_i)$ can be regarded as a 
metric measure space by the 
Riemannian distance $d_{g_i}$ and 
the provability measure $\frac{\mu_{g_i}}{\mu_{g_i}(X)}$, 
where $\mu_{g_i}$ is the Riemannian measure.

On the Riemannian manifolds, the Laplace operators 
can be also defined for the 
differential forms or 
smooth section of vector bundle 
with the metric connections. 
In \cite{Lott-form2002}, Lott discussed with 
the eigenvalues of the Laplacian acting on differential forms 
on the Riemannian manifolds with bounded sectional curvatures and 
diameters, but may be collapsing. 
He also discussed with the eigenvalues of 
the Dirac operators in 
\cite{Lott-Dirac2002}. 
In \cite{Honda2017spectral},
Honda showed the continuity of 
the eigenvalues of 
the Hodge Laplacian acting on the $1$-forms 
with respect to the measured Gromov-Hausdorff 
topology under the assumption that the Ricci 
curvatures are bounded from below and above, 
the diameters are bounded from above and 
the volumes are bounded from below. 
He also showed the continuity of 
the eigenvalues of connection Laplacian 
acting on the tensor bundle 
of cotangent bundles and tangent bundles under 
the same assumptions. 

In another direction, 
Kasue showed the convergence 
of the eigenvalues of the connection Laplacians on vector bundles over the Riemannian 
manifolds who have the uniform estimates of heat kernels in \cite{Kasue2011}. 
He considered the vector bundle 
$E_\rho=P\times_\rho V$ with $G$-connection $\nabla$ on a smooth Riemannian 
manifold $(X,g)$. 
Here, $P$ is a principal $G$-bundle, $(\rho,V)$ is a 
representation of $G$, $P\times_\rho V$ is the associate bundle 
and $\nabla=\nabla^A$ is the connection on $E_\rho$ 
induced by a connection form $A\in\Omega^1(P,\mathfrak{g})$. 
Then the connection Laplacian 
$\nabla^*\nabla$ acting on $\Gamma(E)$ can be identified 
with the Laplacian of the certain Riemannian metric 
$h=h(g,A,\sigma)$ on $P$ 
determined by $g,A$ and an bi-invariant metric 
$\sigma$ on $G$. 

In this article, we also consider the continuity of 
the connection Laplacian on $E_\rho$. 
The difference between \cite{Kasue2011} and this 
article is the topologies on the space of 
the vector bundles with connections. 
In this article we consider the asymptotically $G$-equivariant 
measured Gromov-Hausdorff topology on the space of 
compact metric spaces with isometric $G$-actions, 
which was also introduced in \cite{FuakayaYamaguchi1994} and 
\cite{Lott-Dirac2002}\cite{Lott-form2002}. 
Then we can define the convergence of 
the sequence of 
$\{ (P_i,A_i)\}_i$, where $P_i$ is a principal 
$G$-bundle over a Riemannian manifold $(X_i,g_i)$ 
and $A_i$ is a $G$-connection on  $P_i$. 
Denote by $\lambda^\rho_{i,j}$ the $j$-th eigenvalue 
of ${\nabla^{A_i}}^*\nabla^{A_i}$ acting on $\Gamma(P_i\times_\rho V)$. 
The aim of this article is 
to show that if $\{ (P_i,A_i)\}_i$ is the convergent 
sequence, then the limit $\lim_{i\to\infty}\lambda^\rho_{i,j}$ 
exists for any $j$ and $(\rho,V)$. 
As a consequence, we have the following theorem. 
\begin{thm}
Let $G$ be a compact Lie group and 
$(\rho,V)$ is a real unitary representation of $G$. 
For any $\kappa\in\R,D,N>0$ and $j\in\Z_{\ge 0}$, 
there exist constants $0\le C_j$ 
depending only on 
$n,\kappa,D,N,j,G,\rho,V$ such that 
$\lim_{j\to\infty} C_j=\infty$ and 
the following holds. 
For any closed Riemannian manifold $(X,g)$ 
of dimension $n$ and principal 
$G$-bundle $\pi\colon P\to X$ with 
the $G$-connection $A$ such that 
\begin{align*}
&{\rm dim}\, X = n,\quad 
{\rm Ric}_g \ge \kappa g,\quad 
{\rm diam}\, X \le D,\\
&\| (d^{\nabla^A})^* F^A \|_{L^\infty} \le N, \quad
\| F^A \|_{L^\infty} \le N,
\end{align*}
we have 
\begin{align*}
\lambda_j^{\nabla^{A}} \ge C_j,
\end{align*}
where $F^A\in\Omega^2(X,E_\rho)$ is 
the curvature form of $A$. 
\label{main 1-1}
\end{thm}
We also have the uniform upper estimate 
of $\lambda_j^{\nabla^{A}}$ in 
Theorem \ref{bdd of eigenvalues 2}, 
however, we need additional assumptions.

Kasue has also shown the convergence of 
the eigenvalues of connection Laplacians 
under the spectral topology in \cite{Kasue2011}. 
He also argued the relation between the spectral 
topology and the Gromov-Hausdorff topology in 
\cite{Kasue2002}

Next we apply Theorem 
\ref{main 1-1} to a compact 
K\"ahler manifold 
$(X,\omega)$ and 
a holomorphic vector bundle $E\to X$ 
with hermitian metric $h$. 
Then there is the unique connection 
$\nabla$ which preserves the holomorphic structure and $h$. 
If $\nabla$ is a Hermitian-Einstein 
connection, then it is a Yang-Mills connection.

\begin{thm}
For any $\kappa,\mu\in\R,D>0$ 
and $n,r\in\Z_{>0}$, 
there exist a positive integer $N$ 
depending only on 
$\kappa,\mu,D,n,r$ such that 
the following holds. 
For any closed K\"ahler manifold 
$(X,\omega)$ and holomorphic vector 
bundle
$E\to X$ with 
the Hermitian-Einstein connection $\nabla$ 
such that 
\begin{align*}
&\dim X = n,\quad 
{\rm rk}\, E=r,\quad 
\frac{2\pi \dim X \cdot 
c_1(E)\cdot [\omega]^{n-1}}
{{\rm rk}(E)[\omega]^n} = \mu,\\
&{\rm Ric}_\omega \ge \kappa \omega,\quad 
{\rm diam}_\omega X \le D,\quad
\| F^\nabla \|_{L^\infty} \le D,
\end{align*}
we have 
\begin{align*}
\dim H^0(X,E) \le N.
\end{align*}
\label{main 1-2}
\end{thm}

Theorem \ref{main 1-1} is proved as follows. 
To discuss with the connection Laplacian $\nabla^*\nabla$, 
we need the relation between 
\begin{align*}
\nabla^*\nabla\colon \Gamma(E_\rho)\to\Gamma(E_\rho)
\end{align*}
and 
the Laplacian $\Delta$ of $h$ on $P$. 
In Section \ref{G-str} we review the relation 
between the sections of $E_\rho$ and 
$G$-equivariant $V$-valued smooth functions on $P$. 
Next we construct the Riemannian metric $h$ on 
$P$ from $(X,g),A$ then show that 
$\nabla^*\nabla$ is related to the Laplacian of 
$h$ along \cite{Kasue2011}\cite{Lott-Dirac2002} in Sections \ref{principal metric} and 
\ref{eigen decomp}.
In Section \ref{curvature} we compute the Ricci curvature 
of $h$, and we will see that the boundedness of 
$\| F^A \|_{L^\infty}$ and 
$\| (d^{\nabla^A})^* F^A \|_{L^\infty}$ in the assumption of 
Theorem \ref{main 1-1} is necessary to give the lower bound of 
the Ricci curvature of $h$. 
In Section \ref{G-str on metric space}, 
we introduce the notion of the asymptotically $G$-equivariant measured 
Gromov-Hausdorff convergences for the sequences 
of compact metric measure spaces with isometric $G$-actions, 
then show that if the sequence is precompact 
with respect to the measured 
Gromov-Hausdorff topology, then it is also precompact 
with respect to the asymptotically $G$-equivariant measured 
Gromov-Hausdorff topology. 
In Sections \ref{sec conv} and \ref{sec spec}, we review the results for 
the convergence of spectral structures 
on some metric measure spaces
along \cite{Cheeger-Colding3}\cite{KuwaeShioya2003} 
and apply them to our situation, 
then we see that these arguments are compatible with 
the $G$-actions. 
As a consequence, we obtain 
Theorem \ref{main 1-1} in Section \ref{sec main}. 
In Section \ref{holomorphic}, 
we apply Theorem \ref{main 1-1} 
to the case of holomorphic vector bundles 
on compact K\"ahler manifolds then obtain 
Theorem \ref{main 1-2}.

\paragraph{Acknowledgments.}
The author would like to thank 
Professor Shouhei Honda for 
useful conversations and telling 
him the related works. 
He would also like to thank 
Professor Atsushi Kasue for his comments and advices. 
He would like to thank Daisuke Kazukawa 
for his comments.

\section{Connections on principal bundles}\label{G-str}
Let $X$ be a smooth manifold and $\pi\colon 
P\to X$ 
be a principal $G$-bundle. 
A $G$-connection on $P$ is a differential form 
$A\in\Omega^1(P)\otimes\geh$ satisfying 
\begin{align*}
A_u(\xi^\sharp_u) = \xi,\quad R_\gamma^*A = {\rm Ad}_{\gamma^{-1}}A
\end{align*}
for all $u\in P$, $\xi\in\geh$ and $\gamma\in G$. 
Here, $\xi^\sharp\in\mathcal{X}(P)$ is the vector 
field generated by $\xi\in\geh$, defined by 
\begin{align*}
\xi^\sharp_u:=\left.\frac{d}{dt}\right|_{t=0}u\exp(t\xi).
\end{align*}
The $G$-connection 
determines the horizontal distribution 
$H=\{ H_u\}_{u\in P}$ by $H_u:={\rm Ker}\,A_u
\subset T_uP$. 
The curvature form $F^A\in\Omega^2(P)
\otimes \geh$ is defined by 
\begin{align*}
F^A := dA|_{H} = dA + \frac{1}{2}[A\wedge A].
\end{align*}

\subsection{Local trivialization}
Let $(U,x^1,\ldots,x^n)$ be a 
local coordinate on 
$X$ and we suppose that $P|_U = \pi^{-1}(U)=U\times G$. 
Let $\hat{v}_u\in H_u$ be the horizontal lift of 
$v_x\in T_xX$. 
Now we have $T_{(x,\gamma)}P = T_xX\oplus T_\gamma G$, 
and let $e_1,\ldots,e_k\in \geh$ be a basis. 
Then $\{ \partial_i = \frac{\partial}{\partial x^i}, 
e_\alpha^\sharp\}$ becomes 
a basis of $T_{(x,\gamma)}P$, where $(e_\alpha^\sharp)_\gamma=(L_\gamma)_* e_\alpha$. 
Let $A_i^\alpha\colon U\times G\to \R$ be 
defined by 
\begin{align*}
A_i^\alpha(x,\gamma) e_\alpha := A_{x,\gamma}((\partial_i)_x).
\end{align*}
Then we have 
\begin{align*}
A_i^\alpha(x,\gamma)e_\alpha &= {\rm Ad}_{\gamma^{-1}} \{ A_i^\alpha(x,1)e_\alpha\},\\
(\hat{\partial}_i)_{(x,\gamma)} 
&= (\partial_i)_x - A_i^\alpha (x,\gamma) (e_\alpha^\sharp)_\gamma.
\end{align*}
Now, we fix a real unitary representation 
$\rho:G\to O(V)$ for a vector space 
$V$ with inner product, and have 
\begin{align*}
\Omega_B^k(P,\rho,V) 
= \{ \tau\in \Omega^k(P)\otimes V;\, R_\gamma^*\tau=\rho(\gamma^{-1})\tau,\  
\iota_{\xi^\sharp}\tau = 0\}.
\end{align*}
For the associate vector bundle 
$E_\rho:=P\times_\rho V$, 
we have the natural identification 
\begin{align*}
\Omega^k(X,E_\rho) \cong \Omega_B^k(P,\rho,V).
\end{align*}
In particular, the correspondence between 
$\tau\in\Gamma(X,E_\rho)$ and 
$\hat{\tau}\in\Omega_B^0(P,\rho,V)$ is 
given by 
\begin{align*}
\tau(x) = u\times_\rho \hat{\tau}(u)\quad 
(u\in\pi^{-1}(x)).
\end{align*}
Denote by $\nabla^A$ the covariant derivative 
on $E_\rho$, then $\nabla \tau$ corresponds 
to $d\hat{\tau}|_H$ under the identification 
$\Omega^1(X,E_\rho) \cong
\Omega_B^1(P,\rho,V)$. 
Under the direct decomposition 
$T_uP = H_u \oplus {\rm Ker}\,d\pi_u$, 
we have 
\begin{align*}
d\hat{\tau} = d\hat{\tau}|_H + d\hat{\tau}|_{{\rm Ker}\,d\pi}.
\end{align*}
For $\xi\in\geh$ we have 
\begin{align*}
d\hat{\tau}_u(\xi^\sharp_u)
&= \left.\frac{d}{dt}\right|_{t=0}
\hat{\tau}(u\exp(t\xi))\\
&= \left.\frac{d}{dt}\right|_{t=0}
\rho(\exp(-t\xi))\hat{\tau}(u)
= -\rho_*(\xi)\hat{\tau}(u),
\end{align*}
hence we obtain 
\begin{align}
d\hat{\tau} = d\hat{\tau}|_H - \rho_*\circ A(\cdot)
\hat{\tau},\label{ortho decomp}
\end{align}

\section{Riemannian metric on $P$}\label{principal metric}
In this section we describe the 
relation between the rough Laplacian on the 
associate bundle $E_\rho$ and the Laplacian on $P$ 
for the certain metric along \cite{Kasue2011}\cite{Lott-Dirac2002}.

Fix an ${\rm Ad}_G$-invariant metric $\sigma$ on 
$\mathfrak{g}$ and define a Riemannian metric 
\begin{align}
h = h(g,A,\sigma)\label{def principal metric}
\end{align}
on $P$ by 
\begin{align*}
h_{ij} &:= h(\hat{\partial_i},\hat{\partial_j}) = 
g(\partial_i,\partial_j)
=:g_{ij},\\
h_{i\alpha} &= h(\hat{\partial_i}, e_\alpha^\sharp) = 0,\\
h_{\alpha\beta} &:= h(e_\alpha^\sharp, e_\beta^\sharp) 
= \sigma(e_\alpha,e_\beta)
=: \sigma_{\alpha\beta}.
\end{align*}
By the decomposition \eqref{ortho decomp}, 
we have
\begin{align*}
h^{-1}(d\hat{\tau}_0,d\hat{\tau}_1)
&= h^{-1}(d\hat{\tau}_0|_H,d\hat{\tau}_1|_H)
+ \sigma^{\alpha\beta} \langle 
\rho_*(e_\alpha)\hat{\tau}_0, 
\rho_*(e_\beta)\hat{\tau}_1
\rangle_V\\
&= g^{ij}\langle\nabla_{\partial_i}^A \tau_0, \nabla_{\partial_j}^A \tau_1\rangle_{E_\rho} 
- \langle \sigma^{\alpha\beta} 
\rho_*(e_\beta)\rho_*(e_\alpha)\hat{\tau}_0, 
\hat{\tau}_1
\rangle_V.
\end{align*}
Therefore 
by integrating on $P$ we have 
\begin{align*}
\int_P d^*d\hat{\tau}_0\cdot\hat{\tau}_1\, d{\rm vol}_h
&= 
\int_P \langle (\nabla^A)^*\nabla^A \tau_0, \tau_1\rangle_{E_\rho} d{\rm vol}_h \\
&\quad\quad 
- \int_P \langle \sigma^{\alpha\beta} 
\rho_*(e_\beta)\rho_*(e_\alpha)\hat{\tau}_0, 
\hat{\tau}_1
\rangle_V d{\rm vol}_h
\end{align*}
for any $\tau_0,\tau_1\in\Gamma(X,E_\rho)$, 
which gives 
\begin{align*}
\Delta^h \hat{\tau}
= \widehat{\Delta^A \tau}
- \sigma^{\alpha\beta} 
\rho_*(e_\beta)\rho_*(e_\alpha)\hat{\tau},
\end{align*}
where $\Delta^h=d^*d$ is the 
Laplacian of $h$ acting on the functions and 
$\Delta^A= (\nabla^A)^*\nabla^A$ 
is the rough Laplacian acting on the 
sections of $E_\rho$.

Here, 
\begin{align*}
\sigma^{\alpha\beta}\rho_*(e_\alpha) \rho_*(e_\beta) \colon V\to V
\end{align*}
is a $G$-equivariant map 
whose eigenvalues are real and nonpositive.
Consequently, 
if $(\rho,V)$ is an irreducible representation, 
then by the following lemma we may write 
\begin{align*}
\sigma^{\alpha\beta}\rho_*(e_\alpha) \rho_*(e_\beta) 
= -\chi_{\sigma,\rho}
\end{align*}
for some nonnegative number $\chi_{\sigma,\rho}$, 
called Casimir invariant, determined by 
$(\rho,V)$ and $\sigma$.

\begin{lem}
Let $(\rho,V)$ be an irreducible 
real $G$-representation and 
$\Phi\colon V\to V$ be a $G$-equivariant 
linear map which has at least one real eigenvalue. 
Then there is $a\in\R$ such that 
$\Phi = a\cdot {\rm id}_V$. 
\end{lem}
\begin{proof}
Let $a\in\R$ be an eigenvalue 
of $\Phi$ and 
$V(a)\subset V$ be the eigenspace associate with  
$a\in\R$. 
Since $V(a)$ is a subrepresentation of $V$, 
$V(a)=V$ holds since $V$ is irreducible. 
\end{proof}

\section{Eigenspaces}\label{eigen decomp}
Let $X,P,G,E_\rho,\rho,V,h,A$ be as above. 
For $\lambda\in \R$ put 
\begin{align*}
W^{E_\rho}(\lambda) &:= \{ s\in\Gamma(E_\rho);\, \Delta^A s = \lambda s\},\\
W^h(\lambda) &:= \{ f\in C^\infty(P);\, \Delta^h f = \lambda f\},
\end{align*}
then we can see 
\begin{align*}
(W^h(\lambda)\otimes V)^G
&= \{ f\in C_B^\infty(P,\rho,V);\, \Delta^h f = \lambda f\}.
\end{align*}
By the previous section we obtain an isomorphism 
\[ 
\left.
\begin{array}{ccc}
(W^h(\lambda)\otimes V)^G & \stackrel{\cong}{\longrightarrow} & 
W^{E_\rho}(\lambda - \chi_{\sigma,\rho}) \\
\rotatebox{90}{$\in$} & & \rotatebox{90}{$\in$} \\
\hat{\tau} & \longmapsto & u\times_\rho \hat{\tau}.
\end{array}
\right.
\]

\section{Curvature}\label{curvature}
In this section we compute the curvature of 
$h(g,A,\sigma)$. 
Define $F_{ij}^\alpha\in C^\infty(P|_U)$ by 
\begin{align*}
F_{ij}^\alpha e_\alpha := F^A(\hat{\partial}_i,\hat{\partial}_j)\in\geh,  
\end{align*}
and let $\Gamma_{ij}^k$ be the Christoffel 
symbols of $g$. 
By Section \ref{principal metric}, 
we have 
\begin{align*}
[\hat{\partial}_i, \hat{\partial}_j]
&= -F_{ij}^\alpha e_\alpha^\sharp, \quad
[\hat{\partial}_i, e_\beta^\sharp]
=0, \quad
[e_\alpha^\sharp, e_\beta^\sharp] 
= [e_\alpha, e_\beta]^\sharp,\\
\nabla_{\hat{\partial}_i} \hat{\partial}_j 
&= \Gamma_{ij}^k \hat{\partial}_k 
-\frac{1}{2} F_{ij}^\alpha e_\alpha^\sharp,\quad 
\nabla_{e_\alpha^\sharp} e_\beta^\sharp
= \frac{1}{2}[e_\alpha, e_\beta]^\sharp,\\
\nabla_{\hat{\partial}_i} e_\beta^\sharp
&= \nabla_{e_\beta^\sharp} \hat{\partial}_i
= \frac{g^{kh}\sigma(F_{ih}, e_\beta) }{2}
\hat{\partial}_k
= \frac{g^{kh}F_{ih}^\alpha \sigma_{\alpha\beta} }{2}
\hat{\partial}_k, 
\end{align*}
where $F_{ij} = F_{ij}^\alpha e_\alpha$. 
Moreover the 2nd Bianchi identity yields 
\begin{align*}
0 
= dF^A(\hat{\partial}_i, \hat{\partial}_j, \hat{\partial}_k)
= \hat{\partial}_i(F_{jk}) - \hat{\partial}_j(F_{ik}) +  \hat{\partial}_k(F_{ij}).
\end{align*}
Now we denote by $\hat{R}$ the 
curvature tensor of $h$, and by $R$ that of $g$. 
Then we have
\begin{align*}
\hat{R}(\hat{\partial}_i, \hat{\partial}_j) \hat{\partial}_k
&= R_{ijk}^l \hat{\partial}_l
+\frac{\hat{\partial}_k(F_{ij}^\alpha) -F_{il}^\alpha\Gamma_{jk}^l
- F_{lj}^\alpha\Gamma_{ik}^l}{2}e_\alpha^\sharp \\
&\quad\quad +
\frac{g^{lh}\sigma_{\alpha\beta}}{4}
(2F_{ij}^\alpha F_{kh}^\beta
- F_{jk}^\alpha F_{ih}^\beta
- F_{ki}^\alpha F_{jh}^\beta)\hat{\partial}_l,\\
\hat{R}(\hat{\partial}_i, e_\alpha^\sharp) \hat{\partial}_j
&= \frac{g^{lh}\sigma_{\alpha\beta}}{2}
(\hat{\partial}_i(F_{jh}^\beta) -F_{jk}^\beta\Gamma_{ih}^k
- F_{kh}^\beta\Gamma_{ij}^k)\hat{\partial}_l\\
&\quad\quad 
-\frac{g^{kh}F_{jh}^\beta F_{ik}^\mu\sigma_{\alpha\beta}e_\mu^\sharp}{4}
+\frac{e_\alpha^\sharp(F_{ij}^\beta)e_\beta^\sharp}{2}
+\frac{F_{ij}^\beta[e_\alpha,e_\beta]^\sharp}{4},\\
\hat{R}(\hat{\partial}_i, e_\alpha^\sharp) 
e_\beta^\sharp
&= \frac{g^{kh}(F_{ih}^\mu\sigma(e_\mu,[e_\alpha,e_\beta]) - 2e_\alpha^\sharp(F_{ih}^\mu)\sigma_{\mu\beta})}{4}\hat{\partial}_k\\
&\quad\quad 
-\frac{g^{kh} F_{ih}^\mu\sigma_{\mu\beta}g^{lp}F_{kp}^\delta\sigma_{\delta\alpha}}{4}\hat{\partial}_l,\\
\hat{R}(e_\alpha^\sharp, e_\beta^\sharp) 
e_\mu^\sharp
&= \frac{[e_\mu,[e_\alpha,e_\beta]]^\sharp}{4}.
\end{align*}
Here, we have 
\begin{align*}
e_\alpha^\sharp(F_{ij}^\beta)e_\beta
= -F_{ij}^\beta[e_\alpha,e_\beta],
\quad 
e_\alpha^\sharp(F_{ij}^\beta)e_\beta^\sharp
= -F_{ij}^\beta[e_\alpha,e_\beta]^\sharp.
\end{align*}
By putting 
\begin{align*}
(\nabla F)_{kij}^\alpha
= \hat{\partial}_k(F_{ij}^\alpha) -F_{il}^\alpha\Gamma_{jk}^l
- F_{lj}^\alpha\Gamma_{ik}^l,
\end{align*}
The 2nd Bianchi identity implies 
\begin{align*}
(\nabla F)_{ijk}^\alpha + (\nabla F)_{jki}^\alpha 
+ (\nabla F)_{kij}^\alpha = 0,
\end{align*}
therefore we obtain 
\begin{align*}
\hat{R}(\hat{\partial}_i, \hat{\partial}_j) \hat{\partial}_k
&= R_{ijk}^l \hat{\partial}_l
+\frac{(\nabla F)_{kij}^\alpha}{2}e_\alpha^\sharp \\
&\quad\quad +
\frac{g^{lh}\sigma_{\alpha\beta}}{4}
(2F_{ij}^\alpha F_{kh}^\beta
- F_{jk}^\alpha F_{ih}^\beta
- F_{ki}^\alpha F_{jh}^\beta)\hat{\partial}_l,\\
\hat{R}(\hat{\partial}_i, e_\alpha^\sharp) \hat{\partial}_j
&= \frac{g^{lh}\sigma_{\alpha\beta}
(\nabla F)_{ijh}^\beta}{2}\hat{\partial}_l
-\frac{g^{kh}F_{jh}^\beta F_{ik}^\mu\sigma_{\alpha\beta}e_\mu^\sharp}{4}
\\
&\quad\quad -\frac{F_{ij}^\beta[e_\alpha,e_\beta]^\sharp}{4},\\
\hat{R}(\hat{\partial}_i, \hat{\partial}_j) e_\alpha^\sharp
&= -\frac{g^{lh}\sigma_{\alpha\beta}
(\nabla F)_{hij}^\beta}{2}\hat{\partial}_l
-\frac{g^{kh}(F_{jh}^\beta F_{ik}^\mu 
- F_{ih}^\beta F_{jk}^\mu)\sigma_{\alpha\beta}e_\mu^\sharp}{4}\\
&\quad\quad 
-\frac{F_{ij}^\beta[e_\alpha,e_\beta]^\sharp}{2},\\
\hat{R}(\hat{\partial}_i, e_\alpha^\sharp) 
e_\beta^\sharp
&= -\frac{g^{kh} F_{ih}^\mu\sigma(e_\mu,[e_\alpha,e_\beta])}{4}\hat{\partial}_k
-\frac{g^{kh} F_{ih}^\mu\sigma_{\mu\beta}g^{lp}F_{kp}^\delta\sigma_{\delta\alpha}}{4}\hat{\partial}_l,\\
\hat{R}(e_\alpha^\sharp, e_\beta^\sharp) 
\hat{\partial}_i
&= \frac{g^{kh} F_{ih}^\mu\sigma(e_\mu,[e_\alpha,e_\beta])}{2}\hat{\partial}_k\\
&\quad\quad 
+ \frac{g^{kh} F_{ih}^\mu g^{lp}F_{kp}^\delta
(\sigma_{\mu\beta}\sigma_{\delta\alpha}
- \sigma_{\mu\alpha}\sigma_{\delta\beta})}{4}\hat{\partial}_l,\\
\hat{R}(e_\alpha^\sharp, e_\beta^\sharp) 
e_\mu^\sharp
&= \frac{[e_\mu,[e_\alpha,e_\beta]]^\sharp}{4}.
\end{align*}
Hence the Ricci curvature of $h$ is given by 
\begin{align*}
\hat{{\rm Ric}}(\hat{\partial}_j,\hat{\partial}_k)
&= {\rm Ric}_{jk} 
-\frac{g^{ih}\sigma_{\alpha\beta}F_{ki}^\alpha 
F_{jh}^\beta}{2},\\
\hat{{\rm Ric}}(\hat{\partial}_j,e_\beta^\sharp)
&= -\frac{g^{ih}\sigma_{\beta\mu}
(\nabla F)_{hij}^\mu}{2},\\
\hat{{\rm Ric}}(e_\beta^\sharp, e_\mu^\sharp)
&= 
\frac{g^{kh} F_{hi}^\alpha\sigma_{\alpha\mu}g^{ip}F_{kp}^\delta\sigma_{\delta\beta}}{4}
+ \frac{\sigma^{\alpha\delta}
\sigma([e_\alpha,e_\beta],[e_\delta,e_\mu])}{4}.
\end{align*}
Now, define $F^*F\in\Gamma({\rm Symm}_2(H^*))\otimes 
{\rm Symm}_2(\mathfrak{g})$ by 
\begin{align*}
F^*F = g^{kl}F_{ik}^\alpha F_{jl}^\beta \hat{\partial}^i
\otimes \hat{\partial}^j \otimes e_\alpha 
\otimes e_\beta
\end{align*}
where $H^* \to P$ is the dual of the horizontal distribution 
$H\to P$ and $\{ \hat{\partial}^i\}_i$ is the dual basis of 
$\{ \hat{\partial}_i\}_i$. 
Note that $\{ (d^\nabla)^* F\}_j^\mu = -g^{ih} (\nabla F)_{hij}^\mu$. 
Then we have 
\begin{align*}
\hat{{\rm Ric}}(\hat{\partial}_j,\hat{\partial}_k)
&= {\rm Ric}_{jk} 
-\frac{(F^*F)_{jk}^{\alpha\beta}\sigma_{\alpha\beta}}{2},\\
\hat{{\rm Ric}}(\hat{\partial}_j,e_\beta^\sharp)
&= \frac{ \{ (d^\nabla)^* F\}_j^\mu\sigma_{\beta\mu}}{2},\\
\hat{{\rm Ric}}(e_\beta^\sharp, e_\mu^\sharp)
&= 
\frac{g^{jk} (F^*F)_{jk}^{\alpha\delta}\sigma_{\alpha\beta}
\sigma_{\delta\mu}}{4}
+ \frac{\sigma^{\alpha\delta}
\sigma([e_\alpha,e_\beta],[e_\delta,e_\mu])}{4}.
\end{align*}

\section{$G$-structures on metric spaces}\label{G-str on metric space}
In Section \ref{principal metric}, we have shown that 
$\pi\colon (P,h)\to (X,g)$ is a Riemannian submersion and 
every $G$-orbit is totally geodesically embedded in $P$ and isometric to 
$(G,\sigma)$. 
Conversely, if $(P,h)$ is a Riemannian manifold with isometric free $G$-action 
for a compact Lie group $G$ and every $G$-orbit is isometric to 
$(G,\sigma)$, then $X=P/G$ is a smooth manifold with a Riemannian metric 
$g$ such that $\pi\colon (P,h)\to (X,g)$ is a Riemannian submersion, 
and the horizontal distribution of $P$ defines a $G$-connection. 
We generalize this picture to the metric spaces.

In this article, $G$-actions on the metric spaces 
always mean 
the right actions, and the maps 
\[ 
\left.
\begin{array}{ccc}
P\times G & \longrightarrow & P \\
\rotatebox{90}{$\in$} & & \rotatebox{90}{$\in$} \\
(u,\gamma) & \longmapsto & u\gamma
\end{array}
\right.
\]
are always supposed to be continuous. 
We denote by $\bar{u}\in P/G$ the equivalence 
class represented by $u\in P$.

\begin{prop}
Let $G$ be a compact topological group, 
and $(P,d)$ be a metric space with isometric continuous 
right $G$-action. 
Then the quotient map $\pi\colon (P,d) \to (P/G,\bar{d})$ is a submetry, 
where $\bar{d}$ is a distance function on $P$ defined by 
$\bar{d}(\bar{u}_0,\bar{u}_1):=\inf_{g\in G}d(u_0,u_1g)$. 
\label{quotient metric}
\end{prop}
\begin{proof}
Let $D(u_0,r):=\{ u_1\in P;\, d(u_1,u)\le r\}$. 
We show $\pi(D(u_0,r)) = D(\bar{u}_0,r)$ for 
all $u_0\in P$ and $r>0$. 
Let $u_1\in D(u_0,r)$. Then $\bar{d}(\bar{u}_0,\bar{u}_1)\le d(u_0,u_1) \le r$, hence we have 
$\pi(u_1)=\bar{u}_1\in D(\bar{u}_0,r)$. 
Conversely, if $\bar{u}_1\in D(\bar{u}_0,r)$, then there exists $g\in G$ such that 
$d(u_0,u_1g)=\bar{d}(\bar{u}_0,\bar{u}_1)\le r$, consequently 
$\bar{u}_1\in\pi(D(u_0,r))$. 
\end{proof}

For a topological space $X$, we denote by 
$C_0(X)$ the set consisting of all continuous functions on $X$ 
with compact support. 

Throughout of this paper 
$(P,d,\nu)$ is called 
{\it metric measure space} if 
$(P,d)$ is a metric space and 
$\nu$ is a Borel Radon measure such that 
$\nu(B(p,r))$ is finite for any $p\in P$ and $r>0$.

\begin{definition}
\normalfont

\mbox{}
\begin{itemize}
\setlength{\parskip}{0cm}
\setlength{\itemsep}{0cm}
 \item[(1)] Let $(P',d')$ and $(P,d)$ be metric spaces. 
A map $\phi:P'\to P$ is an {\it $\varepsilon$-isometry} if 
$(i)$ $|d'(u_0,u_1)-d(\phi(u_0),\phi(u_1))|<\varepsilon$ for any $u_0,u_1\in P'$, 
$(ii)$ $P\subset B(\phi(P'),\varepsilon)$. 
 \item[(2)] Let $\{(P_i,d_i,\nu_i)\}_i$ be a sequence of metric measure spaces. 
A metric measure space 
$(P_\infty,d_\infty,\nu_\infty)$ is 
{\it the measured  Gromov-Hausdorff limit of 
$\{(P_i,d_i,\nu_i)\}_i$} 
if there are positive numbers 
$\{ \varepsilon_i\}_i$ with 
$\lim_{i\to \infty}\varepsilon_i = 0$ and Borel $\varepsilon_i$-isometries 
$\phi_i\colon P_i\to P_\infty$ for every $i$ 
such that ${\phi_i}_*\nu_i\to\nu_\infty$ vaguely as $i\to\infty$, that is, 
\begin{align*}
\lim_{i\to\infty} \int_{P_\infty}f d{\phi_i}_*\nu_i = \int_{P_\infty}f d\nu_\infty
\end{align*}
holds for any $f\in C_0(P_\infty)$.
 \item[(3)] Let $\{(P_i,d_i,\nu_i,p_i)\}_i$ be a sequence of pointed 
metric measure spaces. 
A pointed metric measure space 
$(P_\infty,d_\infty,\nu_\infty,p_\infty)$ is 
{\it the pointed measured Gromov-Hausdorff limit of 
$\{(P_i,d_i,\nu_i,p_i)\}_i$} 
if for any $R>0$ there are positive numbers 
$\{ \varepsilon_i\}_i$ and $\{ R_i\}_i$ with 
$\lim_{i\to \infty}\varepsilon_i = 0$, $\lim_{i\to \infty}R_i=R$ 
and Borel $\varepsilon_i$-isometries 
\begin{align*}
\phi_i\colon (B(p_i,R_i),p_i)\to (B(p_\infty,R),p_\infty)
\end{align*}
for every $i$ 
such that ${\phi_i}_*(\nu_i|_{B(p_i,R_i)}) \to \nu_\infty|_{B(p_\infty,R)}$ 
vaguely.
\end{itemize}
\label{def mGH}
\end{definition}

The following notion 
$(1)$ is the special case of 
\cite[Definition 4.1]{FuakayaYamaguchi1994}. 
\begin{definition}
\normalfont 

Let $G$ be a compact topological group. 
\begin{itemize}
\setlength{\parskip}{0cm}
\setlength{\itemsep}{0cm}
 \item[(1)] 
Let $(P',d')$ and $(P,d)$ be metric spaces with 
isometric $G$-action. 
A map $\phi:P'\to P$ is an {\it $\varepsilon$-$G$-equivariant 
Hausdorff approximation} if $\phi$ is an $\varepsilon$-isometry 
satisfying $d(\phi(u'\gamma),\phi(u')\gamma) < \varepsilon$ 
for any $u'\in P'$ and $\gamma\in G$. 
Moreover if $\phi$ is a Borel map then 
it is called a {\it Borel $\varepsilon$-$G$-equivariant 
Hausdorff approximation}. 
 \item[(2)] Let $\{(P_i,d_i,\nu_i)\}_i$ be a sequence of 
metric measure spaces with isometric 
$G$-action. 
A metric measure space 
$(P_\infty,d_\infty,\nu_\infty)$ with isometric $G$-action is called 
{\it the asymptotically $G$-equivariant measured Gromov-Hausdorff limit of 
$\{(P_i,d_i,\nu_i)\}_i$} 
if there is a Borel 
$\varepsilon_i$-$G$-equivariant 
Hausdorff approximation 
$\phi_i\colon P_i\to P_\infty$ for each $i$ 
such that $\lim_{i\to \infty}\varepsilon_i = 0$ and 
${\phi_i}_*\nu_i\to\nu_\infty$ vaguely as 
$i\to\infty$.
 \item[(3)] Let $\{(P_i,d_i,\nu_i,p_i)\}_i$ be a sequence of pointed 
metric measure spaces with isometric $G$-action. 
$(P_\infty,d_\infty,\nu_\infty,p_\infty)$ is said to be 
{\it the pointed asymptotically $G$-equivariant measured Gromov-Hausdorff limit of 
$\{(P_i,d_i,\nu_i,p_i)\}_i$} 
if $G$ acts on $P_\infty$ isometrically and 
for any $R>0$ there are positive numbers 
$\{ \varepsilon_i\}_i$, $\{ R_i\}_i$ with 
\begin{align*}
\lim_{i\to \infty}\varepsilon_i = 0,\quad \lim_{i\to \infty}R_i=R,
\end{align*}
and Borel $\varepsilon_i$-$G$-equivariant 
Hausdorff approximation 
\begin{align*}
\phi_i\colon (\pi_i^{-1}(B(x_i,R_i)),p_i)\to (\pi_\infty^{-1}(B(x_\infty,R)),p_\infty)
\end{align*}
for every $i$ 
such that ${\phi_i}_*(\nu_i|_{\pi_i^{-1}(B(x_i,R_i))} )
\to \nu_\infty|_{\pi_\infty^{-1}(B(x_\infty,R))}$ 
vaguely. 
Here, $\pi\colon P_i\to P_i/G$ is the quotient map and 
$x_i=\pi_i(p_i)$. 
\end{itemize}
\label{def GmGH}
\end{definition}

\begin{rem}
\normalfont
In Definition \ref{def GmGH}, 
$\phi_i\colon P_i\to P_\infty$ are not required 
to be $G$-equivariant. 
If all of $\phi_i$ are $G$-equivariant, 
then we obtain another topology 
which is already introduced by Lott 
in \cite{Lott-Dirac2002}\cite{Lott-form2002}. 
\end{rem}

We denote by 
\begin{align*}
(P_i,d_i,\nu_i) \mGH (P_\infty, d_\infty, \nu_\infty) 
\quad ({\rm or}\ P_i \mGH P_\infty)
\end{align*}
the pair of a sequence of metric measure spaces $\{(P_i,d_i,\nu_i)\}_i$ 
and its measured Gromov-Hausdorff limit $(P_\infty, d_\infty, \nu_\infty)$. 
Similarly, we write  
\begin{align*}
(P_i,d_i,\nu_i) \GmGH (P_\infty, d_\infty, \nu_\infty) 
\quad ({\rm or}\ P_i \GmGH P_\infty)
\end{align*}
if $G$ acts on 
$(P_i,d_i,\nu_i)$ and $(P_\infty, d_\infty, \nu_\infty)$, and 
$(P_\infty, d_\infty, \nu_\infty)$ is the asymptotically $G$-equivariant 
measured Gromov-Hausdorff limit of $\{ (P_i,d_i,\nu_i)\}_i$.

Similarly, 
\begin{align*}
(P_i,d_i,\nu_i,p_i) &\mGH (P_\infty, d_\infty, \nu_\infty,p_\infty),\\
(P_i,d_i,\nu_i,p_i) &\GmGH (P_\infty, d_\infty, \nu_\infty,p_\infty)
\end{align*}
mean the pointed measured Gromov-Hausdorff convergence 
and the pointed asymptotically $G$-equivariant 
measured Gromov-Hausdorff convergence, respectively.

\begin{rem}
\normalfont 
By \cite[Remark 2.2]{KuwaeShioya2003}, 
$(P_i,d_i,\nu_i,p_i) \mGH (P_\infty,d_\infty,\nu_\infty,p_\infty)$ 
iff there are positive numbers 
$\{ \varepsilon_i\}_i$, $\{ R_i\}_i$, $\{ R'_i\}_i$ 
and Borel $\varepsilon_i$-isometries 
$\phi_i\colon B(p_i,R'_i)\to B(p_\infty,R_i)$ for every $i$ 
such that $\lim_{i\to \infty}\varepsilon_i = 0$, 
$\lim_{i\to \infty}R_i=\lim_{i\to \infty}R'_i=\infty$ and 
\begin{align*}
\phi_i(p_i)=p_\infty,\quad 
\lim_{i\to\infty} \int_{B(p_\infty,R_i)}f d{\phi_i}_*\nu_i = \int_{B(p_\infty,R_i)}f d\nu_\infty
\end{align*}
hold for any $f\in C_0(P_\infty)$. 
\label{equiv pmGH}
\end{rem}

\begin{rem}
\normalfont 
If $(P_i,d_i,\nu_i,p_i) \GmGH (P_\infty,d_\infty,\nu_\infty,p_\infty)$, 
then 
the limit is the 
pointed measured compact Gromov Hausdorff limit 
in the sense of Definition 2.2 in \cite{KuwaeShioya2003}. 
\end{rem}

\begin{prop}
Let $(P,d)$ and $(P',d')$ be metric spaces with isometric 
$G$-action, 
and $\phi:P'\to P$ be an $\varepsilon$-$G$-equivariant 
Hausdorff approximation. 
Then there exists a $2\varepsilon$-isometry 
$\bar{\phi}\colon P'/G\to P/G$. 
Moreover, suppose that 
$\nu$ and $\nu'$ are Borel measures 
on $P$ and $P'$ respectively, 
$\phi$ is a Borel map, 
and there is a Borel section $s'\colon P'/G \to P'$, 
namely, Borel map $s'$ with $\pi'\circ s'={\rm id}_{P'/G}$. 
Then 
$\bar{\phi}$ is also a Borel map and the following holds. 
For any $f\in C_0(P/G)$ and $\varepsilon'>0$ 
there exists $\delta>0$ depending only on $\varepsilon'$ and $f$ 
such that if $\varepsilon\le\delta$ then 
\begin{align*}
\left| \int_X f\, d(\bar{\phi}_*\bar{\nu}') - \int_X f\, d\bar{\nu}\right|
&\le \nu' (P') \varepsilon' 
+ \left| \int_P f \circ\pi\, d\phi_*\nu' 
- \int_P f\circ\pi \, d\nu\right|
\end{align*}
holds, 
where $\bar{\nu} = \pi_*\nu$ and $\bar{\nu}' = \pi'_*\nu'$.
\label{G-eq conv and base conv}
\end{prop}
\begin{proof}
Put $X:=P/G$, $X':=P'/G$. 
First of all we define $\bar{\phi}\colon P'/G\to P/G$ 
such that 
\begin{align*}
|\bar{d}'(\bar{u}_0,\bar{u}_1) - \bar{d}(\bar{\phi}(\bar{u}_0),\bar{\phi}(\bar{u}_1)) |
< 2\varepsilon. 
\end{align*}
Fix $s'\colon P'/G \to P'$ with $\pi'\circ s'={\rm id}_{P'/G}$ 
and put 
\begin{align*}
\bar{\phi}(x) := \overline{\phi( s'(x) )}. 
\end{align*}
Note that $\bar{\phi}$ is a Borel map if so are $\phi$ and $s'$. 

For $x_0,x_1\in P'$, take $\gamma_0\in G$ such that 
\begin{align*}
d(\phi(s'(x_0)),\phi(s'(x_1))\gamma_0) 
= \bar{d}(\bar{ \phi }(x_0),\bar{\phi}(x_1)).
\end{align*}
Then we have 
\begin{align*}
\bar{d}'(x_0,x_1) - \bar{d}(\bar{\phi}(x_0),\bar{\phi}(x_1)) 
&= \bar{d}'(x_0,x_1) - d(\phi(s'(x_0)),\phi(s'(x_1))\gamma_0) \\
&< \bar{d}'(x_0,x_1) - d(\phi(s'(x_0)),\phi(s'(x_1)\gamma_0)) + \varepsilon \\
&< \bar{d}'(x_0,x_1) - d'(s'(x_0),s'(x_1)\gamma_0) 
+ 2\varepsilon \\
&\le 2\varepsilon.
\end{align*}
Take $\gamma_1\in G$ such that 
$d'(s'(x_0),s'(x_1)\gamma_1) = \bar{d}'(x_0,x_1)$. 
Then we have 
\begin{align*}
\bar{d}(\bar{\phi}(x_0),\bar{\phi}(x_1)) - \bar{d}'(x_0,x_1) 
&= \bar{d}(\bar{\phi}(x_0),\bar{\phi}(x_1)) - d'(s'(x_0),s'(x_1)\gamma_1)\\
&< \bar{d}(\bar{\phi}(x_0),\bar{\phi}(x_1)) 
- d(\phi(s'(x_0)),\phi(s'(x_1)\gamma_1)) \\
&\quad\quad+ \varepsilon\\
&< \bar{d}(\bar{\phi}(x_0),\bar{\phi}(x_1)) 
- d(\phi(s'(x_0)),\phi(s'(x_1))\gamma_1) \\
&\quad\quad+ 2\varepsilon\\
&\le 2\varepsilon.
\end{align*}
Next we show $X\subset B(\bar{\phi}(X'),2\varepsilon)$. 
Let $\bar{u}\in X$, then there is $u'\in P'$ 
such that $d(u,\phi(u')) < \varepsilon$. 
Take $\gamma_2\in G$ such that $u'=s'(\bar{u}')\gamma_2$. 
Then we have 
\begin{align*}
\bar{d}(\bar{u},\bar{\phi}(\bar{u}')) 
\le d(u,\phi(s'(\bar{u}'))\gamma_2) 
< d(u,\phi(s'(\bar{u}')\gamma_2)) + \varepsilon <2\varepsilon, 
\end{align*}
which implies $\bar{u}\in B(\bar{\phi}(X'),2\varepsilon)$.

Suppose $\nu$ and $\nu'$ are Borel measures 
on $P$ and $P'$ respectively and 
$f\colon X_\infty\to \R$ is a continuous function. 
Now we have 
\begin{align*}
\left| \int_X f\, d(\bar{\phi}_*\bar{\nu}') - \int_X f\, d\bar{\nu}\right|
&\le \left| \int_{P'} f\circ\bar{\phi}\circ\pi'\, d\nu' 
- \int_{P'} f\circ\pi\circ\phi\, d\nu'\right|\\
&\quad\quad + \left| \int_P f \circ\pi\, d\phi_*\nu' 
- \int_P f\circ\pi \, d\nu\right|\\
&\le \left\| f\circ\bar{\phi}\circ\pi' 
- f\circ\pi\circ\phi \right\|_{L^\infty}\nu' (P')\\
&\quad\quad + \left| \int_P f \circ\pi\, d\phi_*\nu' 
- \int_P f\circ\pi \, d\nu\right|.
\end{align*}
Since the support of $f$ is compact, 
it is uniformly continuous, therefore 
for any $\varepsilon'>0$ there is $\delta>0$ 
which depends only on $\varepsilon'$ and $f$ such that 
$\bar{d} (x_0,x_1)<\delta$ implies $|f(x_0)-f(x_1)|<\varepsilon'$. 
Since 
\begin{align*}
\bar{d} (\bar{\phi}\circ\pi'(u'),\pi\circ\phi(u'))
&= \inf_{\gamma\in G} d(\phi (s'(\bar{u}')),\phi(u')\gamma) \\
&< \inf_{\gamma\in G} d(\phi (s'(\bar{u}')), \phi (u'\gamma)) 
+ \varepsilon \\
&< \inf_{\gamma\in G} d'(s'(\bar{u}'), u'\gamma) 
+ 2\varepsilon = 2\varepsilon,
\end{align*}
hence if $\varepsilon$ is not more than $\frac{\delta}{2}$ 
then 
\begin{align*}
\left| \int_X f\, d(\bar{\phi}_*\bar{\nu}') - \int_X f\, d\bar{\nu}\right|
&\le \nu' (P') \varepsilon' 
+ \left| \int_P f \circ\pi\, d\phi_*\nu' 
- \int_P f\circ\pi \, d\nu\right|
\end{align*}
holds. 
\end{proof}

As a consequence of Proposition 
\ref{G-eq conv and base conv}, 
we obtain the following results.

\begin{cor}
If $(P_i,d_i,\nu_i) \GmGH (P_\infty,d_\infty,\nu_\infty)$, $P_i$ and $P_\infty$ are relatively 
compact and 
every $P_i\to P_i/G$ has a Borel sections, then 
$(P_i/G,\bar{d}_i,\bar{\nu}_i) \mGH (P_\infty/G,\bar{d}_\infty,\bar{\nu}_\infty)$. 
\label{bundle to base}
\end{cor}

Next we generalize Corollary \ref{bundle to base} 
to the pointed noncompact spaces.

\begin{definition}
\normalfont 
A metric space is said to be 
{\it proper} if all open balls $B(p,r)$ are relatively compact.  
\end{definition}
It is easy to see that subsets of proper metric spaces are compact iff they are 
closed and bounded.

\begin{lem}
Let $(P,d)$ be a metric space, 
$G$ be a compact topological group and 
$G$ acts on $P$ isometrically. 
Assume that $(P/G,\bar{d})$ is proper. 
Then the quotient map $\pi\colon P\to P/G$ 
is a proper map iff $P$ is proper. 
\end{lem}
\begin{proof}
Suppose $\pi$ is proper. 
Since $\pi$ is submetry, 
$\pi(B(u,r))=B(\bar{u},r)$ holds for any $u\in P$ 
and $r>0$. 
Then one can see 
$B(u,r)\subset \pi^{-1}(B(\bar{u},r))$, 
hence $B(u,r)$ is relatively compact. 
Next we suppose $P$ is proper, 
then it is locally compact. 
For any compact subset $K\subset P/G$ 
we take an open cover 
$\pi^{-1}(K)\subset \bigcup_\alpha 
U_\alpha$. 
By the locally compactness 
of $P$ we obtain another open cover
$\pi^{-1}(K)\subset \bigcup_\beta V_\beta$ 
such that every 
$V_\beta$ is relatively compact and contained in some $U_\alpha$. 
Since $\pi$ is an open map, 
$\{ \pi(V_\beta)\}_\beta$ is 
an open cover of $K$, 
therefore we have 
$K\subset \bigcup_{i=1}^N V_{\beta_i}$ for 
some $\beta_1,\ldots,\beta_N$. 
Denote by $\varphi\colon P\times G\to P$ 
the continuous map which defines the $G$-action 
on $P$. Then one can see that 
$\pi^{-1}(K)$ is contained in 
$\varphi(\bigcup_i V_{\beta_i}\times G)$, which is 
relatively compact.
Since $\pi^{-1}(K)$ is closed, it is compact.
\end{proof}

\begin{cor}
If $(P_i,d_i,\nu_i,p_i) \GmGH (P_\infty,d_\infty,\nu_\infty,p_\infty)$, 
every $P_i\to P_i/G$ has a Borel sections, 
$P_i/G$ and $P_\infty/G$ are proper, 
$\pi_i$ and $\pi_\infty$ are proper and 
$\nu_\infty(K)<\infty$ holds for any compact 
subset $K\subset P_\infty$, 
then 
$(P_i/G,\bar{d}_i,\bar{\nu}_i) \mGH (P_\infty/G,\bar{d}_\infty,\bar{\nu}_\infty)$. 
\label{bundle to base pt}
\end{cor}
\begin{proof}
By applying Proposition \ref{G-eq conv and base conv} for $P'=B(p_i,r_i)$ and 
$P=B(p_\infty,r_\infty)$, we can see that 
$(P_i,d_i,p_i)$ converges to 
$(P_\infty,d_\infty,p_\infty)$ with respect to 
pointed Gromov-Hausdorff topology. 
To show the vague convergence of 
measures by using the inequality 
in Proposition \ref{G-eq conv and base conv}, 
we have to show that $f\circ \pi_\infty$ has compact support and 
$\limsup_{i\to \infty}\nu_i(B(p_i,r_i))
<\infty$ holds if $r_i\to r_\infty$. 
Since we suppose that 
$f\in C_0(B(\bar{p}_\infty,r))$ and $\pi_\infty$ 
is proper, then the support of 
$f\circ \pi_\infty$ is compact. 
Next we take $f\in C_0(P_\infty)$ 
such that $f=1$ on $\bigcup_i \phi_i(B(p_i,r_i))$ 
and $f\le 1$ everywhere, 
then one can see 
\begin{align*}
\limsup_{i\to \infty}\nu_i(B(p_i,r_i))
\le \nu_\infty({\rm supp}(f)).
\end{align*}
\end{proof}

\begin{prop}
Let $(P_i,d_i) \GH (P_\infty,d_\infty)$, that is, 
there are $\varepsilon_i$-isometry 
$\phi_i\colon P_i\to P_\infty$ for each $i$ and 
$\lim_{i\to\infty}\varepsilon_i = 0$ holds. 
Suppose that a compact topological group $G$ acts on 
every $P_i$ isometrically and 
$P_\infty$ is relatively compact. 
Moreover assume that the family of continuous maps 
$\{ F_{i,u}\}_{i\in\Z_{\ge 0},u\in P_i}$, where 
\[ 
\left.
\begin{array}{cccc}
F_{i,u}\colon & G\times G & \longrightarrow & \R \\
& \rotatebox{90}{$\in$} & & \rotatebox{90}{$\in$} \\
& (\gamma,\gamma') & \longmapsto & d_i(u\gamma, u\gamma'),
\end{array}
\right.
\]
is equicontinuous. 
For any decreasing sequence 
$\{ \hat{\varepsilon}_k\}_k$ with 
$\lim_{k\to \infty}\hat{\varepsilon}_k = 0$, 
there are a subsequence 
$\{ (P_{i_k},d_{i_k})\}_k$ and 
an isometric $G$-action on $(P_\infty,d_\infty)$ 
such that $\phi_{i_k}$ is an 
$\hat{\varepsilon}_k$-$G$-equivariant 
Hausdorff approximation. 
\label{cptness 1}
\end{prop}

\begin{proof}
Put 
\begin{align*}
G_0:=\{ \gamma\in G;\, u\gamma=u \mbox{ for 
all }u\in X_i\mbox{ and }i\in\Z_{\ge 0}\},
\end{align*}
then $G_0$ is a closed normal subgroup of $G$. 
First of all we induce the metric on the quotient group $G_0\backslash G$ compatible with its quotient topology. 
For every $u\in P_i$, let $P_{i,u}$ be 
the metric space isometric to $P_i$ and put 
\begin{align*}
\mathcal{X}:=\prod_{i\in\Z_{\ge 0},u\in P_i} P_{i,u}
\end{align*}
be the metric space whose distance function is given by 
\begin{align*}
d_\mathcal{X}((p_{i,u})_{i,u},(q_{i,u})_{i,u}):=
\sup_{i,u}d_i(p_{i,u},q_{i,u}).
\end{align*}
Define the injective map 
$f_\mathcal{X}\colon G_0\backslash G \to \mathcal{X}$ 
by $f_\mathcal{X}(\gamma):=(u\gamma)_{i,u}$. 
Here, we denote by $\gamma$ the right coset 
in $G_0\backslash G$ represented by 
$\gamma$ for the brevity. 
Then the induced metric 
$d_G:=f_\mathcal{X}^*d_\mathcal{X}$ on 
$G_0\backslash G$ is given by 
\begin{align*}
d_G(\gamma,\gamma')
= \sup_{i,u}d(u\gamma,u\gamma') 
= \sup_{i,u}F_{i,u}(\gamma,\gamma').
\end{align*}
Denote by $\mathcal{O}$ the quotient topology 
on $G_0\backslash G$ and denote 
by $\mathcal{O}_{d_G}$ the topology induced by 
$d_G$. 
Since $\{ F_{i,u}\}_{i,u}$ is equicontinuos, 
$d_G\colon G_0\backslash G \times 
G_0\backslash G \to \R$ is continuous 
with respect to $\mathcal{O}$, 
hence $\mathcal{O}_{d_G}$ is weaker than or 
equal to $\mathcal{O}$ then one can see 
$(G_0\backslash G, \mathcal{O}_{d_G})$ is compact 
and $f_\mathcal{X}(G_0\backslash G)$ is closed in 
$\mathcal{X}$. 
Now we can see that 
$f_\mathcal{X}\colon (G_0\backslash G,\mathcal{O}) \to 
\mathcal{X}$ is continuous by 
the following reason. 
If $x\in \mathcal{X}\setminus 
f_\mathcal{X}(G_0\backslash G)$, 
there exists $\delta>0$ such that 
$B(x,\delta)\subset \mathcal{X}\setminus 
f_\mathcal{X}(G_0\backslash G)$, 
hence $f_\mathcal{X}^{-1}(B(x,\delta))=\emptyset$. 
If $x\in f_\mathcal{X}(G_0\backslash G)$ then 
there exists $\gamma_0\in G$ such that 
$f_\mathcal{X}(\gamma_0)=x$ and 
we have $f_\mathcal{X}^{-1}(B(x,\delta))=
d_G(\gamma_0,\cdot)^{-1}((-\infty,\delta))$, 
which is an open set. 
Thus we can see that 
$f_\mathcal \colon
(G_0\backslash G,\mathcal{O}) \to f_\mathcal{X}(G_0\backslash G)$ 
is a bijective continuous map from a 
compact space to a Hausdorff space, hence it is 
a homeomorphism, which implies 
$\mathcal{O}=\mathcal{O}_{d_G}$.

Since $P_\infty$ and 
$G_0\backslash G$ 
are compact 
metric spaces, they are separable. 
Let 
\begin{align*}
\{ u^{(\alpha)}\}_{\alpha=0}^\infty \subset P_\infty,
\quad \{ \gamma^{(\beta)}\}_{\beta=0}^\infty \subset G_0\backslash G
\end{align*}
be countable dense subsets. 

By the compactness of $P_\infty$ and 
$G_0\backslash G$, 
we may assume that there are increasing sequences of integers 
$0=K_0<K_1<K_2<\cdots$ and 
$0=L_0<L_1<L_2<\cdots$ such that 
\begin{align*}
\bigcup_{\alpha=K_m}^{K_{m+1} -1} 
B\left(u^{(\alpha)}, \frac{1}{2^m}\right)
 = P_\infty,
\quad \bigcup_{\beta=L_m}^{L_{m+1}-1} 
B\left(\gamma^{(\beta)}, \frac{1}{2^m}\right) 
= G_0\backslash G.
\end{align*}

First of all we define 
$u^{(\alpha)}\cdot\gamma^{(\beta)}\in P_\infty$. 
Take $u^{(\alpha)}_i\in P_i$ such that 
\begin{align*}
d_\infty(\phi_i(u^{(\alpha)}_i),u^{(\alpha)})
<\varepsilon_i.
\end{align*}
Since $P_\infty$ is relatively compact, 
there is a subsequence of $\{ \phi_i(u^{(\alpha)}_i\gamma^{(\beta)})\}_i$ converging to 
some points in $P_\infty$. 
By repeating this procedure for 
\begin{align*}
(\alpha,\beta)=(0,0),(1,0),(0,1),(2,0),(1,1),(0,2),\ldots
\end{align*}
and by the diagonal argument, 
there is $\{ i_k\}_{k=0}\subset \{ i=0,1,2,\ldots\}$ 
such that 
$\{ \phi_{i_k}(u^{(\alpha)}_{i_k}\gamma^{(\beta)})\}_{i_k}$ 
converges to the limit. 
We put 
\begin{align*}
u^{(\alpha)}\gamma^{(\beta)}
:= \lim_{k\to\infty} \phi_{i_k}(u^{(\alpha)}_{i_k}\gamma^{(\beta)}), 
\end{align*}
then we obtain a map 
\begin{align}
\{ u^{(\alpha)}\}_{\alpha}\times 
\{ \gamma^{(\beta)}\}_{\beta}
\to P_\infty,\quad 
(u^{(\alpha)},\gamma^{(\beta)})\mapsto 
u^{(\alpha)}\gamma^{(\beta)}.\label{action map}
\end{align}
By taking a subsequence 
$\{ i_{0,k}\}_k$ of $\{ i_k\}_k$, 
we may suppose that 
\begin{align*}
d_\infty(\phi_{i_{0,k}}(u^{(\alpha)}_{i_{0,k}}\gamma^{(\beta)}), 
u^{(\alpha)}\gamma^{(\beta)})
< \frac{1}{2^k}
\end{align*}
for any $0\le \alpha< K_1$, 
$0\le \beta< L_1$ and $k\ge 0$. 
We can take a subsequence 
$\{ i_{m,k}\}_k$ of $\{ i_{m-1,k}\}_k$ 
inductively such that 
\begin{align*}
d_\infty(\phi_{i_{m,k}}(u^{(\alpha)}_{i_{m,k}}\gamma^{(\beta)}), 
u^{(\alpha)}\gamma^{(\beta)})
< \frac{1}{2^k}
\end{align*}
for any $K_m\le \alpha< K_{m+1}$, 
$L_m\le \beta< L_{m+1}$ and $k\ge 0$. 
By replacing $i_k$ by $i_{k,k}$, we may assume that 
\begin{align*}
d_\infty(\phi_{i_k}(u^{(\alpha)}_{i_k}\gamma^{(\beta)}), 
u^{(\alpha)}\gamma^{(\beta)})
< \frac{1}{2^k}
\end{align*}
for any $K_m\le \alpha< K_{m+1}$, 
$L_m\le \beta< L_{m+1}$ and $k\ge m$

Next we show the continuity of the above map. 
We have 
\begin{align*}
d_\infty(u^{(\alpha)}\gamma^{(\beta)}, 
u^{(\alpha')}\gamma^{(\beta')})
&= d_\infty(\lim_{k\to\infty} \phi_{i_k}(u^{(\alpha)}_{i_k}\gamma^{(\beta)}), 
\lim_{k\to\infty} \phi_{i_k}(u^{(\alpha')}_{i_k}\gamma^{(\beta')}))\\
&= \lim_{k\to\infty} d_\infty(\phi_{i_k}(u^{(\alpha)}_{i_k}\gamma^{(\beta)}), 
\phi_{i_k}(u^{(\alpha')}_{i_k}\gamma^{(\beta')}))\\
&\le \limsup_{k\to\infty} \left( 
d_{i_k}(u^{(\alpha)}_{i_k}\gamma^{(\beta)}, 
u^{(\alpha')}_{i_k}\gamma^{(\beta')}) 
+ \varepsilon_{i_k} \right)\\
&= \limsup_{k\to\infty} 
d_{i_k}(u^{(\alpha)}_{i_k}\gamma^{(\beta)}, 
u^{(\alpha')}_{i_k}\gamma^{(\beta')}) \\
&\le \limsup_{k\to\infty} \Big( 
d_{i_k}(u^{(\alpha)}_{i_k}\gamma^{(\beta)}, 
u^{(\alpha)}_{i_k}\gamma^{(\beta')}) \\
&\quad\quad \quad\quad 
\quad + d_{i_k}(u^{(\alpha)}_{i_k}\gamma^{(\beta')}, 
u^{(\alpha')}_{i_k}\gamma^{(\beta')})
\Big) \\
&= \limsup_{k\to\infty} \Big( 
d_{i_k}(u^{(\alpha)}_{i_k}\gamma^{(\beta)}, 
u^{(\alpha)}_{i_k}\gamma^{(\beta')}) 
+ d_{i_k}(u^{(\alpha)}_{i_k}, u^{(\alpha')}_{i_k})
\Big) \\
&\le d_G(\gamma^{(\beta)},\gamma^{(\beta')}) 
+ d_\infty(u^{(\alpha)}, u^{(\alpha')}).
\end{align*}
which gives the continuity of 
\eqref{action map}. 
Then we can extend the map to 
\begin{align}
P_\infty \times G
\to P_\infty,\quad 
(u,\gamma)\mapsto 
u\gamma.\label{action map 2}
\end{align}
continuously. 
Next we show that for any $u_{i_k}\in P_{i_k}$ and $\gamma\in G$ with 
$\phi_{i_k}(u_{i_k})\to u$, 
we have $\phi_{i_k}(u_{i_k}\gamma)\to u\gamma$. 
Take $u^{(\alpha)}$ and $\gamma^{(\beta)}$ such that 
$d_\infty(u^{(\alpha)}, u)$ and $d_(\gamma^{(\beta)},\gamma)$ are small. 
Then one can see 
\begin{align*}
d_{\infty}(\phi_{i_k}(u^{(\alpha)}_{i_k}\gamma^{(\beta)}),\phi_{i_k}(u_{i_k}\gamma))
&< d_{i_k}(u^{(\alpha)}_{i_k}\gamma^{(\beta)}, u_{i_k}\gamma) 
+ \varepsilon_{i_k}\\
&\le d_{i_k}(u^{(\alpha)}_{i_k}\gamma^{(\beta)}, u^{(\alpha)}_{i_k}\gamma)\\
&\quad\quad 
+ d_{i_k}(u^{(\alpha)}_{i_k}\gamma, u_{i_k}\gamma) + \varepsilon_{i_k}\\
&\le d_G(\gamma^{(\beta)},\gamma)
+ d_{i_k}(u^{(\alpha)}_{i_k}, u_{i_k})
+ \varepsilon_{i_k}\\
&\le d_G(\gamma^{(\beta)},\gamma)
+ d_\infty (\phi_{i_k}(u^{(\alpha)}_{i_k}), \phi_{i_k}(u_{i_k})) + 2\varepsilon_{i_k}\\
&\le d_G(\gamma^{(\beta)},\gamma)
+ d_\infty (u^{(\alpha)},u)\\
&\quad\quad + d_\infty (\phi_{i_k}(u_{i_k}),u)
+ 3\varepsilon_{i_k},
\end{align*}
therefore we obtain
\begin{align}
d_{\infty}(u\gamma,\phi_{i_k}(u_{i_k}\gamma))
&< 
d_{\infty}(u\gamma,u^{(\alpha)}\gamma^{(\beta)})
+ d_{\infty}(u^{(\alpha)}\gamma^{(\beta)}, 
\phi_{i_k}(u^{(\alpha)}_{i_k}\gamma^{(\beta)}))\notag\\
&\quad\quad
+ d_G(\gamma^{(\beta)},\gamma)+ d_\infty (u^{(\alpha)},u)\notag\\
&\quad\quad
+ d_\infty (\phi_{i_k}(u_{i_k}),u)
+ 3\varepsilon_{i_k}\notag\\
&\le 
d_{\infty}(u^{(\alpha)}\gamma^{(\beta)}, 
\phi_{i_k}(u^{(\alpha)}_{i_k}\gamma^{(\beta)}))
+ 2d_G(\gamma^{(\beta)},\gamma)\notag\\
&\quad\quad
+ 2d_\infty (u^{(\alpha)},u)
+ d_\infty (\phi_{i_k}(u_{i_k}),u)
+ 3\varepsilon_{i_k}.
\label{ineq action}
\end{align}

Now, choose 
$\alpha,\beta$ such that 
$K_k\le \alpha< K_{k+1}$, 
$L_k\le \beta< L_{k+1}$ and 
$d_\infty (u^{(\alpha)},u)<\frac{1}{2^k}$, 
$d_G (\gamma^{(\beta)},\gamma)<\frac{1}{2^k}$. 
Then by \eqref{ineq action} we have 
\begin{align}
d_{\infty}(u\gamma,
\phi_{i_k}(u_{i_k}\gamma))
&\le \frac{5}{2^k} + 3\varepsilon_{i_k} 
+ d_\infty (\phi_{i_k}(u_{i_k}),u),
\label{ineq action 2}
\end{align}
which implies 
$\lim_{k\to\infty}\phi_{i_k}(u_{i_k}\gamma) = u\gamma$.

Next we have to show $(u\gamma)\gamma'=u(\gamma\gamma')$ 
for any $u\in P_\infty$ and $\gamma,\gamma'\in G$. 
Take $u_{i_k}\in P_{i_k}$ such that 
$\phi_{i_k}(u_{i_k})\to u$ as $k\to\infty$. 
Then we have 
$\phi_{i_k}(u_{i_k}\gamma) \to u\gamma$, hence 
$\phi_{i_k}(u_{i_k}\gamma\gamma') \to (u\gamma)\gamma'$. 
Obviously, $\phi_{i_k}(u_{i_k}\gamma\gamma') \to u(\gamma\gamma')$ also holds, 
therefore we obtain $(u\gamma)\gamma'=u(\gamma\gamma')$. 

One can see that the action on $P_\infty$ is isometric 
since 
\begin{align*}
d_\infty (u_0\gamma, u_1\gamma)
&= \lim_{k\to \infty} d_\infty(\phi_{i_k}(u_{0,i_k}\gamma), 
\phi_{i_k}(u_{1,i_k}\gamma))\\
&= \lim_{k\to \infty} d_{i_k}(u_{0,i_k}\gamma, 
u_{1,i_k}\gamma)\\
&= \lim_{k\to \infty} d_{i_k}(u_{0,i_k}, 
u_{1,i_k})\\
&= \lim_{k\to \infty} d_\infty(\phi_{i_k}(u_{0,i_k}), \phi_{i_k}(u_{1,i_k}))
= d_\infty (u_0, u_1).
\end{align*}
where $\phi_{i_k}(u_{0,i_k})\to u_0$, 
$\phi_{i_k}(u_{1,i_k})\to u_1$.

By \eqref{ineq action 2}, we have 
\begin{align*}
d_\infty(\phi_{i_k}(u'\gamma),\phi_{i_k}(u')\gamma)
&\le \frac{5}{2^k} + 3\varepsilon_{i_k},
\end{align*}
for some $\beta$ with $d_G(\gamma^{(\beta)},\gamma)<\frac{1}{2^k}$, 
which implies that $\phi_{i_k}$ 
is an $\frac{5}{2^k} + 3\varepsilon_{i_k}$-$G$-equivariant 
Hausdorff approximation. 
\end{proof}

\begin{rem}
\normalfont
Suppose $P_i$ is compact. 
The following map 
\[ 
\left.
\begin{array}{ccc}
P_i\times G\times G & \longrightarrow & \R \\
\rotatebox{90}{$\in$} & & \rotatebox{90}{$\in$} \\
(u,\gamma,\gamma') & \longmapsto & d_i(u\gamma, u\gamma'),
\end{array}
\right.
\]
is continuous, hence uniformly continuous. 
Then $\{ F_{i,u}\}_{u}$ is always equicontinuous 
for any fixed $i$. 
\end{rem}

As a consequence of Proposition \ref{cptness 1}, 
we also obtain the following.
\begin{prop}
Let $(P_i,d_i,\nu_i)$ and $(P_\infty,d_\infty,\nu_\infty)$ be metric 
measure spaces 
and all of $(P_i,d_i)$ have isometric $G$-actions satisfying the 
assumption of Proposition \ref{cptness 1}, 
and let 
\begin{align*}
(P_i,d_i,\nu_i) \mGH (P_\infty,d_\infty,\nu_\infty).
\end{align*}
Then there exists a subsequence 
$\{ (P_{i_k},d_{i_k},\nu_{i_k})\}_k$ and 
isometric $G$-action on $P_\infty$ such that 
\begin{align*}
(P_{i_k},d_{i_k},\nu_{i_k}) \GmGH (P_\infty,d_\infty,\nu_\infty).
\end{align*}
\label{cptness 1-1}
\end{prop}

We also have the noncompact version 
of Proposition \ref{cptness 1-1} as follows. 
\begin{prop}
Let $(P_i,d_i,\nu_i,p_i)$ be 
proper pointed metric measure spaces 
with isometric $G$-actions for 
$i\in\N\cup\{ \infty\}$
such that 
\begin{align*}
(P_i,d_i,\nu_i,p_i) \mGH 
(P_\infty,d_\infty,\nu_\infty,p_\infty).
\end{align*}
Assume that 
$\{ F_{i,u}\}_{i\in\Z_{\ge 0},u\in B(p_i,r_i)}$ is 
equicontinuous for any 
convergent sequence $\{ r_i\}_i$. 
Then there exists a subsequence 
$\{ (P_{i_k},d_{i_k},\nu_{i_k},p_{i_k})\}_k$ such that 
\begin{align*}
(P_{i_k},d_{i_k},\nu_{i_k},p_{i_k}) \GmGH (P_\infty,d_\infty,\nu_\infty,p_\infty).
\end{align*}
\label{cptness 1-2}
\end{prop}
\begin{proof}
Take a sequence $\{ r_m\}_m$ 
such that $\lim_{m\to\infty} r_m=\infty$ 
and sequences $\{ r_{m,i}\}_i$ such that 
$\lim_{i\to\infty}r_{m,i} = r_m$. 
Denote by $\pi_i\colon P_i\to P_i/G$ the 
quotient maps. 
By applying Proposition \ref{cptness 1} 
to 
\begin{align*}
\pi_i^{-1}(B(\bar{p}_i,r_{m,i})) \GH 
\pi_\infty^{-1}(B(\bar{p}_\infty,r_m))
\end{align*}
inductively and combining with the diagonal argument, 
we have the assertion. 
\end{proof}

\begin{rem}
\normalfont
Kasue showed the precompactness of the 
families of 
the principal $G$-bundles equipped with the spectral distances in 
\cite{Kasue2011}, 
and he discussed the relation between the spectral distances 
and the Riemannian distances in \cite{Kasue2002}. 
\end{rem}

Any smooth closed 
Riemannian manifold $(X,g)$ can be regarded 
as a metric measure space by the Riemannian distance 
$d_g$ and the measure $\frac{\mu_g}{\mu_g(X)}$, 
where $\mu_g$ is the 
Riemannian measure. 
For a principal $G$-bundle $\pi:P\to X$ and a $G$-connection 
$A$ on $P$, let 
$| F^A |^2(x):= \langle F^A(x),F^A(x)\rangle_{g_x\otimes \sigma}$, 
where $\sigma$ is a ${\rm Ad}_G$-invariant metric on $G$. 
Similarly we can define $| (d^{\nabla^{A}})^* F^{A} |^2(x)$. 
\begin{thm}
Let $\{ (X_i,g_i)\}_{i=0}^\infty$ be a sequence of closed Riemannian manifolds 
with 
\begin{align*}
{\rm dim}\, X_i = n,\, 
{\rm Ric}_{g_i} \ge \kappa g_i,\, 
{\rm diam}\, X_i \le D
\end{align*}
for some constants $n,\kappa,D$ independent of $i$, 
$G$ be a compact Lie group and 
$\pi_i:P_i\to X_i$ be principal $G$-bundle with 
$G$-connection $A_i$ 
satisfying 
\begin{align*}
\sup_{x\in X_i} | (d^{\nabla^{A_i}})^* F^{A_i} |(x) < N, \quad
\sup_{x\in X_i} | F^{A_i} |(x) < N
\end{align*}
for a constant $N>0$ independent of $i$. 
Then there exists a subsequence $\{ i_k\}_k$ and 
a metric measure space $(P_\infty,d_\infty,\nu_\infty)$ with the isometric 
$G$-action such that 
\begin{align*}
(P_{i_k},h(g_{i_k},A_{i_k},\sigma)) \GmGH (P_\infty,d_\infty,\nu_\infty),\quad 
(X_{i_k},g_{i_k}) \mGH (P_\infty/G,\bar{d}_\infty,\bar{\nu}_\infty)
\end{align*}
where $h(g_i,A_i,\sigma)$ is the metric on $P_i$ defined by 
\eqref{def principal metric}. 
\label{precompact}
\end{thm}
\begin{proof}
By the assumption, 
the Ricci curvatures of 
$h(g_i,A_i,\sigma)$ are bounded below by 
the constant independent of $i$, 
hence, 
$\{(P_i,h(g_i,A_i,\sigma))\}_i$
is precompact with respect to the measured 
Gromov-Hausdorff topology. 
Moreover, the functions $F_{i,u}$ in the 
assumption of Proposition \ref{cptness 1} 
satisfies $F_{i,u}(\gamma, \gamma')
\le d_G(\gamma, \gamma')$, 
where $d_G$ is the distance 
induced by $\sigma$, accordingly 
$\{ F_{i,u}\}_{i,u}$ is equicontinuous. 
Therefore, by Proposition 
\ref{cptness 1}, there exists a convergent 
subsequence of 
$\{(P_i,h(g_i,A_i,\sigma))\}_i$ with respect to 
the asymptotically $G$-equivariant measured 
Gromov-Hausdorff topology. 
Finally, to apply Corollary \ref{bundle to base}, 
it suffices to show that all 
principal $G$-bundles $\pi\colon P\to X$ 
over a compact smooth manifold 
has a Borel section. 
Since $X$ is compact, 
there is a finite open covering 
$\{ U_\alpha\}_{\alpha=1}^N$  
of $X$ such that 
$P|_{U_\alpha}\to U_\alpha$ are trivial bundles. 
Put $U_{\alpha}^0:=U_{\alpha}$, 
$U_{\alpha}^1:=X\setminus U_{\alpha}$ and 
$V_\sigma:=\bigcap_{\alpha=1}^N 
U_{\alpha}^{\sigma(\alpha)}$ for a map 
$\sigma\colon \{ 1,\ldots,N\}\to \{ 0,1\}$. 
Then we have $X=\bigsqcup_\sigma V_\sigma$ 
and every 
$P|_{V_\sigma}\to V_\sigma$ has a continuous section $s_\sigma\colon 
V_\sigma\to P|_{V_\sigma}$. 
Since $V_\sigma$ are Borel sets, we have the 
Borel section of $P\to X$ by gluing 
$\{ s_\sigma\}_\sigma$. 
\end{proof}

Similarly we have the followings by Proposition 
\ref{cptness 1-2}. 

\begin{thm}
Let $\{ (X_i,g_i)\}_{i=0}^\infty$ be a sequence of complete Riemannian manifolds 
with 
\begin{align*}
{\rm dim}\, X_i = n,\, 
{\rm Ric}_{g_i} \ge \kappa g_i
\end{align*}
for some constants $n,\kappa$ independent of $i$, 
$G$ be a compact Lie group and 
$\pi_i:P_i\to X_i$ be principal $G$-bundle with 
$G$-connection $A_i$ 
satisfying 
\begin{align*}
\sup_{x\in X_i} | (d^{\nabla^{A_i}})^* F^{A_i} |(x) < N, \quad
\sup_{x\in X_i} | F^{A_i} |(x) < N
\end{align*}
for a constant $N>0$ independent of $i$. 
Let $p_i\in P_i$. 
Then there exists a subsequence $\{ i_k\}_k$ and 
a pointed metric measure space $(P_\infty,d_\infty,\nu_\infty,p_\infty)$ with the isometric 
$G$-action such that 
\begin{align*}
(P_{i_k},h(g_{i_k},A_{i_k},\sigma),p_{i_k}) &\GmGH (P_\infty,d_\infty,\nu_\infty,p_\infty),\\
(X_{i_k},g_{i_k},\bar{p}_{i_k}) &\mGH (P_\infty/G,\bar{d}_\infty,\bar{\nu}_\infty,\bar{p}_\infty)
\end{align*}
where $h(g_i,A_i,\sigma)$ is the metric on $P_i$ defined by 
\eqref{def principal metric}. 
\label{precompact noncpt}
\end{thm}

\section{Convergence of eigenfunctions}\label{sec conv}
Any closed Riemannian manifold $(X,g)$ can be canonically 
regarded as the metric measure space $(X,d_g,\frac{\mu_g}{\mu_g(X)})$, 
where $d_g$ is the Riemannian distance of $g$ and 
$\mu_g$ is the Riemannian measure. 
Denote by $\mathcal{M}(n,\kappa,D)$ 
the set consisting of 
isometric classes of metric measure spaces 
coming from closed Riemannian manifolds 
$(X,g)$ with 
\begin{align*}
{\rm dim}\, X = n,\, 
{\rm Ric}_{g} \ge \kappa g,\, 
{\rm diam}\, X \le D, 
\end{align*}
and let $\overline{\mathcal{M}(n,\kappa,D)}$ be the 
closure with respect to the measured Gromov-Hausdorff 
distance. 

We denote by $\mathcal{M}(n,\kappa)$ 
the set consisting of pointed metric measure spaces 
$(X,d_g,\frac{\mu_g}{\mu_g(B(p,1))},p)$, 
where $(X,g)$ is a complete Riemannian manifold with 
\begin{align*}
{\rm dim}\, X = n,\, 
{\rm Ric}_{g} \ge \kappa g,
\end{align*}
$p\in X$ and $B(p,1)$ is the geodesic ball centered at $p$ 
of radius $1$. 
Let $\overline{\mathcal{M}(n,\kappa)}$ be the 
closure with respect to the pointed measured Gromov-Hausdorff 
distance.

\subsection{Convergence of spectral structures}\label{kuwae shioya}
In \cite{KuwaeShioya2003}, 
Kuwae and Shioya introduced the notion of 
spectral structures for the Laplacian which enabled us 
to treat the convergence of eigenvalues in the systematic way. 
In this subsection we review the framework developed 
in \cite{KuwaeShioya2003}. 

Let $H_i$ be Hilbert spaces over $\mathbb{K}=\R$ 
or $\C$ 
for $i\in\Z_{\ge 0}\cup\{ \infty \}$, $\mathcal{C}\subset H_\infty$ 
be a dense subspace and 
$\Phi_i\colon \mathcal{C}\to H_i$ be linear operators 
which satisfy 
\begin{align}
\lim_{i\to \infty}\| \Phi_i(u)\|_{H_i}=\| u\|_{H_\infty}
\label{conv str}
\end{align}
for any $u\in\mathcal{C}$. 

\begin{definition}[\cite{KuwaeShioya2003}]
\normalfont
Let $u_i\in H_i$ for $i\in\Z_{\ge 0}\cup\{ \infty \}$. 
\begin{itemize}
\setlength{\parskip}{0cm}
\setlength{\itemsep}{0cm}
 \item[(1)] A sequence $\{ u_i\}_i$ {\it converges to $u_\infty$ strongly} as $i\to \infty$ 
if there exists a sequence $\{ \tilde{u}_k\}_k\subset H_\infty$ tending to 
$u_\infty$ such that 
\begin{align*}
\lim_{k\to \infty}\limsup_{i\to \infty}\| \Phi_i(\tilde{u}_k) - u_i\|_{H_i} = 0.
\end{align*}
 \item[(2)] A sequence $\{ u_i\}_i$ {\it converges to $u_\infty$ weakly} as $i\to \infty$ 
if \begin{align*}
\lim_{i\to \infty}\langle u_i,v_i\rangle_{H_i} = 0
\end{align*}
holds for any $( v_i)_i\in \prod_{i\in \Z\cup\{ \infty\}} H_i$ such that 
$v_i\to v_\infty$ strongly. 
\end{itemize}
\label{def spec}
\end{definition}

For a Hilbert space $H$, 
let $A\colon \mathcal{D}(A)\to H$ be a self-adjoint 
linear operator on $H$, where $\mathcal{D}(A)$ is the domain 
of $A$, and suppose $\mathcal{E}$ 
is given by $\mathcal{E}(u)
:=\langle Au, u\rangle_{H}$ 
for $u\in \mathcal{D}(A)$ and $\mathcal{E}(u):=\infty$ 
for $u\in H\backslash\mathcal{D}(A)$. 
Moreover we assume that $\mathcal{E}$ is 
{\it closed}, namely, 
$\mathcal{D}(A)$ is complete 
with respect to the norm 
defined by $\|u\|_{\mathcal{E}}:=\sqrt{\| u\|_H^2+\mathcal{E}(u)}$. 
The spectral structure $\Sigma$ generated by $A$ is defined by 
\begin{align*}
\Sigma := (A,\mathcal{E},E,\{ T_t\}_{t\ge 0}, \{ R_\zeta\}_{\zeta\in\rho(A)}), 
\end{align*}
where $E$ is the spectral measure of $A$, 
$T_t:=e^{-tA}$, $R_\zeta=(\zeta-A)^{-1}$ and 
$\rho(A)$ is the resolvent set of $A$.

\begin{definition}[\cite{KuwaeShioya2003}]
\normalfont
A sequence of closed quadratic forms 
$\{ \mathcal{E}_i\colon H_i \to \overline{\R}\}_i$ 
{\it Mosco converges to 
$\mathcal{E}_\infty\colon H_\infty \to \overline{\R}$ 
} as $i\to \infty$ 
if 
\begin{itemize}
\setlength{\parskip}{0cm}
\setlength{\itemsep}{0cm}
 \item[(1)] $\mathcal{E}_\infty(u_\infty)
\le \liminf_{i\to\infty}\mathcal{E}_i(u_i)$ for any $\{ u_i\}_i$
with $u_i\to u_\infty$ weakly,
 \item[(2)] for any $u_\infty\in H_\infty$ there exists $\{ u_i\}_i$ 
strongly converging to $u_\infty$ such that 
$\mathcal{E}_\infty(u_\infty) = \lim_{i\to\infty}\mathcal{E}_i(u_i)$. 
\end{itemize}
Moreover, 
$\{ \mathcal{E}_i\}_i$ 
{\it compactly converges to 
$\mathcal{E}_\infty$ 
} as $i\to \infty$ 
if 
\begin{itemize}
\setlength{\parskip}{0cm}
\setlength{\itemsep}{0cm}
 \item[(3)] $\{ \mathcal{E}_i\}_i$ 
Mosco converges to 
$\mathcal{E}_\infty$ as $i\to \infty$, 
 \item[(4)] for any $\{ u_i\}_i$ with 
$\limsup_{i\to\infty}(\| u_i\|_{H_i}^2 + \mathcal{E}_i(u_i)) < \infty$, 
there exists a strongly convergent subsequence. 
\end{itemize}
We write 
$\mathcal{E}_i\stackrel{{\rm M}}{\to}\mathcal{E}_\infty$ and 
$\mathcal{E}_i\stackrel{{\rm cpt}}{\to}\mathcal{E}_\infty$, respectively.
\label{def spec2}
\end{definition}

\begin{definition}[\cite{KuwaeShioya2003}]
\normalfont
Let $A_i$ be a self-adjoint nonnegative operator on $H_i$ 
and $\Sigma_i$ be the spectral structure generated by $A_i$. 
Then 
{\it $\{ \Sigma_i\}_i$ strongly converges to 
$\Sigma_\infty$ 
(resp. $\{ \Sigma_i\}_i$ compactly converges to 
$\Sigma_\infty$) 
as $i\to\infty$} if $\mathcal{E}_i\Mosco 
\mathcal{E}_\infty$ 
(resp. $\mathcal{E}_i\cpt \mathcal{E}_\infty$) 
as $i\to\infty$. 
\end{definition}
The authors of \cite{KuwaeShioya2003} 
have shown that the compact convergence of 
$\{ \mathcal{E}_i\}_i$ is 
equivalent to the certain convergence of the other materials 
consisting of $\Sigma_i$. 
See \cite[Section 2.6]{KuwaeShioya2003} for the details.

The definitions of the notions in 
Definition \ref{def spec} and \ref{def spec2} 
depend on 
the choice of $\{ \Phi_i\}_i$ with 
\eqref{conv str}, however, 
we can replace it with other 
$\{ \hat{\Phi}_i\colon \hat{\mathcal{C}}\to H_i\}_i$ 
by the following lemma. 
\begin{lem}
Let $\hat{\mathcal{C}}\subset \mathcal{C}$ be 
dense subspaces of $H_\infty$. 
If $\Phi_i\colon \mathcal{C}\to H_i$ and 
$\hat{\Phi}_i\colon \hat{\mathcal{C}}\to H_i$ 
satisfy \eqref{conv str} and 
$\lim_{i\to\infty} \| \Phi_i(u) 
- \hat{\Phi}_i(u)\|_{H_i}=0$ 
for any $u\in\hat{\mathcal{C}}$, 
then $u_i\to u_\infty$ strongly with respect to 
$\{ \Phi_i\}_i$ iff 
$u_i\to u_\infty$ strongly with respect to 
$\{ \hat{\Phi}_i\}_i$. 
\label{lem conv str}
\end{lem}
\begin{proof}
Suppose $u_i\to u_\infty$ strongly with respect to 
$\{ \hat{\Phi}_i\}_i$, that is, there is a sequence 
$\{ \tilde{u}_k\}_k\subset \hat{\mathcal{C}}$ 
tending to $u_\infty$ such that 
\begin{align*}
\lim_{k\to\infty}\limsup_{i\to\infty}\| \hat{\Phi}_i(\tilde{u}_k) - u_i\|_{H_i}=0.
\end{align*} 
Then we have 
\begin{align*}
\limsup_{i\to\infty}\| \Phi_i(\tilde{u}_k) 
- u_i\|_{H_i}
&\le \limsup_{i\to\infty}\| \hat{\Phi}_i(\tilde{u}_k) - u_i\|_{H_i}\\
&\quad\quad + \limsup_{i\to\infty}\| \Phi_i(\tilde{u}_k) - \hat{\Phi}_i(\tilde{u}_k)\|_{H_i}\\
&=\limsup_{i\to\infty}\| \hat{\Phi}_i(\tilde{u}_k) - u_i\|_{H_i}
\end{align*}
for all $\tilde{u}_k$, hence 
$u_i\to u_\infty$ strongly with respect to 
$\{ \Phi_i\}_i$. 
The converse follows by the similar argument since 
the definition of the strong convergence of $\{ u_i\}$ 
does not depend on the choice of $\{ \tilde{u}_k\}_k$ tending to $u_\infty$. 
\end{proof}

Now let $(P_i,d_i,\nu_i) \in \overline{\mathcal{M}(n,\kappa,D)}$ 
and suppose that $\{ (P_i,d_i,\nu_i)\}_i$ measured Gromov-Hausdorff 
converges to 
$(P_\infty,d_\infty,\nu_\infty)$. 
Then there are Borel $\varepsilon_i$-isometries 
$\phi_i\colon P_i\to P_\infty$ such that 
$\varepsilon_i\to 0$ and ${\phi_i}_*\nu_i\to \nu_\infty$ as 
$i\to \infty$. 
Put $H_i:=L^2(P_i,\nu_i)$, 
$\mathcal{C}:=C(P_\infty)
=\{ f\colon P_\infty\to \R;\, f\mbox{ is continuous.}\}$, $\Phi_i(f):=f\circ\phi_i$ 
and $\Sigma_i$ be the spectral structure generated by $\Delta_{\nu_i}$.

\begin{thm}[\cite{Cheeger-Colding3}\cite{KuwaeShioya2003}]
If $(P_i,d_i,\nu_i) \in \overline{\mathcal{M}(n,\kappa,D)}$ and 
\begin{align*}
(P_i,d_i,\nu_i) \mGH
(P_\infty,d_\infty,\nu_\infty),
\end{align*}
then $\Sigma_i\cpt\Sigma_\infty$. 
\end{thm}

Let $(P_i,d_i,\nu_i,p_i) \in \overline{\mathcal{M}(n,\kappa)}$ 
and suppose that $\{ (P_i,d_i,\nu_i,p_i)\}_i$ 
pointed measured Gromov-Hausdorff 
converges to 
$(P_\infty,d_\infty,\nu_\infty,p_\infty)$. 
Let $\varepsilon_i,R_i,R'_i$ and 
$\phi_i\colon B(p_i,R'_i)\to B(p_\infty,R_i)$ be as 
in Remark \ref{equiv pmGH}. 
Put $H_i:=L^2(P_i,\nu_i)$, 
\begin{align*}
\mathcal{C}:=C_0(P_\infty)
=\{ f\in C(P_\infty);\, {\rm supp}(f)\mbox{ is compact.}\},
\end{align*}
\[ 
\Phi_i(f)(u):=
\left\{ 
\begin{array}{cc}
f\circ\phi_i(u) & u\in B(p_i,R'_i)\\
0 & u\notin B(p_i,R'_i)
\end{array}
\right.
\]
for $f\in \mathcal{C}$ 
and $\Sigma_i$ be the spectral structure generated by $\Delta_{\nu_i}$. 
\begin{thm}[\cite{KuwaeShioya2003}]
If $(P_i,d_i,\nu_i,p_i) \in \overline{\mathcal{M}(n,\kappa)}$ and 
\begin{align*}
(P_i,d_i,\nu_i,p_i) \mGH
(P_\infty,d_\infty,\nu_\infty,p_\infty),
\end{align*}
then $\Sigma_i\str\Sigma_\infty$. 
\end{thm}

Suppose that 
the sequence $(P_i,d_i,\nu_i)$ have isometric $G$-actions, 
then $G$ acts on $H_i=L^2(P_i,\nu_i)$ isometrically. 
In the following subsections we are going to see 
that we can replace $\Phi_i$ by $G$-equivariant maps 
$\hat{\Phi}_i$ which satisfies the assumptions of 
Lemma \ref{lem conv str}.

\subsection{$G$-equivariant maps 
between Hilbert spaces}

In this subsection 
we assume that $G$ is a compact Lie group, 
$(P_i,d_i,\nu_i) \in \mathcal{M}(n,\kappa,D)$ 
and $(P_\infty,d_\infty,\nu_\infty)\in 
\overline{\mathcal{M}(n,\kappa,D)}$ 
have isometric $G$-actions and 
\begin{align*}
(P_i,d_i,\nu_i) \GmGH 
(P_\infty,d_\infty,\nu_\infty).
\end{align*}
Let $\phi_i,\varepsilon_i$ be as in 
Definition \ref{def GmGH}. 

For any $\gamma\in G$, 
put $R_\gamma(u):=u\gamma$, then 
we have 
$(R_\gamma)_*\nu_i = \nu_i$ since 
$\nu_i$ is the Riemannian measure and 
$G$ acts on $(P_i,d_i)$ isometrically. 
Since 
we have 
\begin{align*}
\left|\int_{P_\infty} f d(R_\gamma\circ\phi_i)_*\nu_i
- \int_{P_\infty} f d{\phi_i}_*\nu_i\right|
\le \sup_{u'\in P_i} 
\left| f(\phi_i(u')\gamma) -f(\phi_i(u'\gamma))\right| \nu_i(P_i),
\end{align*}
for any continuous function $f\colon P_\infty\to\R$ and 
the continuous functions on compact metric spaces are 
uniformly continuous, 
then the vague limit of $\{(R_\gamma\circ\phi_i)_*\nu_i\}_i$ 
and $\{ (\phi_i\circ R_\gamma)_*\nu_i\}_i$ should coincides, 
consequently we have 
\begin{align*}
\int_{P_\infty} f d{R_\gamma}_*\nu_\infty = \int_{P_\infty} f d\nu_\infty
\end{align*}
for any continuous function $f$ and $\gamma\in G$. 
By \cite{Cheeger-Colding1997}, $\nu_\infty$ and ${R_\gamma}_*\nu_\infty$ are 
Radon measures, hence we can see 
\begin{align*}
{R_\gamma}_*\nu_\infty = \nu_\infty
\end{align*}
for any $\gamma\in G$. 

Now, let 
\begin{align*}
(P_i,d_i,\nu_i,p_i) \GmGH 
(P_\infty,d_\infty,\nu_\infty,p_\infty) \in 
\overline{\mathcal{M}(n,\kappa)}.
\end{align*}
In this case we also have 
${R_\gamma}_*\nu_\infty = \nu_\infty$ 
for any $\gamma\in G$. 
Similarly as Remark \ref{equiv pmGH}, 
there exist $\varepsilon_i,R_i,R'_i$ 
such that $\lim_{i\to\infty}\varepsilon_i=0$ 
and $\lim_{i\to\infty}R_i = \lim_{i\to\infty}R'_i
=\infty$ and 
$\varepsilon_i$-$G$-equivariant 
Hausdorff approximation 
\begin{align*}
\phi_i\colon \pi_i^{-1}(B(\bar{p}_i,R'_i)\to 
\pi_\infty^{-1}(B(\bar{p}_\infty,R_i))
\end{align*}
such that $\phi_i(p_i) = p_\infty$. 
We also assume that ${\phi_i}_*\nu_i \to \nu_\infty$ 
vaguely. 
Put $H_i:=L^2(P_i,\nu_i)$, 
\begin{align*}
\mathcal{C}:=C_0(P_\infty)
=\{ f\in C(P_\infty);\, {\rm supp}(f)\mbox{ is compact.}\},
\end{align*}
\[ 
\Phi_i(f)(u):=
\left\{ 
\begin{array}{cc}
f\circ\phi_i(u) & u\in \pi_i^{-1}(B(\bar{p}_i,R'_i)\\
0 & u\notin \pi_i^{-1}(B(\bar{p}_i,R'_i)
\end{array}
\right.
\]
for $f\in \mathcal{C}$.
Define $\hat{\Phi}_i\colon \mathcal{C} 
\to H_i$ by 
\begin{align*}
\hat{\Phi}_i(f):=\frac{1}{\mu(G)}
\int_G R_{\gamma^{-1}}^*\Phi_i(R_\gamma^*f)
d\mu,
\end{align*}
where $\mu$ is one of the 
right-invariant measures on $G$.

\begin{prop}
Let $\Phi_i,\hat{\Phi}_i$ be as above. 
Then $\hat{\Phi}_i$ is 
$G$-equivariant and 
\begin{align*}
\lim_{i\to\infty}\| \Phi_i(f) - \hat{\Phi}_i(f)\|_{H_i}
=0
\end{align*}
holds for any $f\in\mathcal{C}$. 
\label{replace of Phi}
\end{prop}
\begin{proof}
It is easy to check that $\hat{\Phi}_i$ is 
$G$-equivariant. 
One can see
\begin{align*}
&\quad \ \| \Phi_i(f) 
- \hat{\Phi}_i(f)\|_{H_i}^2\\
&= \int_{\pi_i^{-1}(B(\bar{p}_i,R'_i)} 
\frac{1}{\mu(G)^2}
\left| \int_G \left( f\circ \phi_i(u)
- f( \phi_i(u\gamma^{-1})\gamma) \right) 
d\mu(\gamma) \right|^2d\nu_i(u).
\end{align*}
Since $\nu_i$ is $G$-invariant, 
we have 
\begin{align*}
&\quad \ \| \Phi_i(f) 
- \hat{\Phi}_i(f)\|_{H_i}^2\\
&= \frac{1}{\mu(G)^2}
\int_{\pi_i^{-1}(B(\bar{p}_i,R'_i)} 
\left| \int_G \left( f\circ \phi_i(u\gamma)
- f( \phi_i(u)\gamma) \right) 
d\mu(\gamma) \right|^2d\nu_i(u).
\end{align*}
Recall that the support of $f$ is compact. 
Then $f$ is uniformly continuous, 
hence for any $\varepsilon>0$ 
there exists $\delta>0$ such that 
$d_\infty(u_0,u_1)<\delta$ implies 
$|f(u_0)-f(u_1)|<\varepsilon$. 
If we take $i_\varepsilon$ such that 
$\varepsilon_i<\delta$ holds 
for any $i\ge i_\varepsilon$, 
then we can see 
$|f\circ \phi_i(u\gamma)-f(\phi_i(u)\gamma)|
<\varepsilon$ 
for any $u\in\pi_i^{-1}(B(\bar{p}_i,R'_i))$, 
$\gamma\in G$ and $i\ge i_\varepsilon$, 
and $|f\circ \phi_i(u\gamma)-f(\phi_i(u)\gamma)|
=0$ for $u\notin \phi_i^{-1}({\rm supp}(f)\cdot G)$, 
which gives 
\begin{align*}
\| \Phi_i(f) 
- \hat{\Phi}_i(f)\|_{H_i}^2
= \varepsilon^2 \nu_i\left( 
\phi_i^{-1}({\rm supp}(f)\cdot G) \right).
\end{align*}
By the same argument in the 
proof of Corollary \ref{bundle to base pt}, 
we have 
\begin{align*}
\limsup_{i\to\infty}\nu_i\left( 
\phi_i^{-1}({\rm supp}(f)\cdot G) \right) < \infty,
\end{align*}
which implies the assertion. 
\end{proof}

\section{Spectral structures with group actions}\label{sec spec}
Let $H_i,\mathcal{C},\Phi_i\colon \mathcal{C}\to H_i$ be 
as in Section \ref{kuwae shioya}, 
and $\Sigma_i$ be the spectral structures generated by 
$A_i\colon\mathcal{D}(A_i)\to H_i$. 
In this section we observe some spectral structures 
induced from $\Sigma_i$. 
\subsection{Tensor products}
Let $V$ be a real vector space of dimension $k$
with a positive definite inner product. 
Then $H_i\otimes V$ are Hilbert spaces whose inner products 
are given by 
\begin{align*}
\left\langle \sum_\alpha u^\alpha\otimes e_\alpha, 
\sum_\alpha v^\alpha\otimes e_\alpha\right\rangle_{H_i\otimes V}
= \sum_\alpha \langle u^\alpha, v^\alpha\rangle_{H_i},
\end{align*}
where $e_1,\ldots,e_k\in V$ is an orthonormal basis and 
$u^\alpha,v^\alpha\in H_i$. 
Then $\mathcal{C}\otimes V$ is dense in $H_\infty\otimes V$ 
and we define $\Phi^V_i\colon \mathcal{C}\otimes V \to H_i\otimes V$ 
by $\Phi^V_i(\sum_\alpha u^\alpha\otimes e_\alpha)
:=\sum_\alpha \Phi_i(u^\alpha)\otimes e_\alpha$. 
Then we can check the followings easily. 
\begin{prop}
Let $u_i^\alpha\in H_i$. 
\begin{itemize}
\setlength{\parskip}{0cm}
\setlength{\itemsep}{0cm}
 \item[$(1)$] $\sum_\alpha u_i^\alpha\otimes e_\alpha \to 
\sum_\alpha u_\infty^\alpha\otimes e_\alpha$ strongly iff 
$u_i^\alpha \to u_\infty^\alpha$ strongly for all $\alpha$. 
 \item[$(2)$] $\sum_\alpha u_i^\alpha\otimes e_\alpha \to 
\sum_\alpha u_\infty^\alpha\otimes e_\alpha$ weakly iff 
$u_i^\alpha \to u_\infty^\alpha$ weakly for all $\alpha$. 
\end{itemize}
\label{tensor conv8.1}
\end{prop}
Let $A_i^V\colon \mathcal{D}(A_i)\otimes V\to H_i\otimes V$ be 
defined by 
\begin{align*}
A_i^V(\sum_\alpha u^\alpha\otimes e_\alpha)
:=\sum_\alpha A_i(u^\alpha)\otimes e_\alpha
\end{align*}
and put 
$\mathcal{E}_i^V(u):=\langle A_i^V(u),u\rangle_{H_i}$, then we have 
$\mathcal{E}_i^V(\sum_\alpha u^\alpha\otimes e_\alpha)
=\sum_\alpha \mathcal{E}_i(u^\alpha)$.

\begin{prop}
\mbox{}

\begin{itemize}
\setlength{\parskip}{0cm}
\setlength{\itemsep}{0cm}
 \item[$(1)$] If $\mathcal{E}_i$ is closed, 
then $\mathcal{E}_i^V$ is closed. 
 \item[$(2)$] If $\mathcal{E}_i\Mosco \mathcal{E}_\infty$ 
(resp. $\mathcal{E}_i\cpt \mathcal{E}_\infty$), 
then $\mathcal{E}_i^V \Mosco 
\mathcal{E}_\infty^V$ 
(resp. $\mathcal{E}_i^V\cpt \mathcal{E}_\infty^V$). 
\end{itemize}
\label{tensor conv}
\end{prop}
\begin{proof}
$(1)$ is obvious. We show $(2)$. 
Suppose $\mathcal{E}_i\Mosco \mathcal{E}_\infty$. 
If $\sum_\alpha u_i^\alpha\otimes e_\alpha \to 
\sum_\alpha u_\infty^\alpha\otimes e_\alpha$ weakly, then by Proposition \ref{tensor conv8.1} we can see 
$u_i^\alpha \to u_\infty^\alpha$ weakly for 
any $\alpha$ and 
\begin{align*}
\mathcal{E}_\infty^V\left(\sum_\alpha u_\infty^\alpha\otimes e_\alpha\right)
= \sum_\alpha \mathcal{E}_\infty
(u_\infty^\alpha)
&\le \sum_\alpha \liminf_{i\to\infty} \mathcal{E}_i
(u_i^\alpha)\\
&\le \liminf_{i\to\infty}\sum_\alpha \mathcal{E}_i
(u_i^\alpha)\\
&= \liminf_{i\to\infty}\mathcal{E}_i^V\left(\sum_\alpha u_\infty^\alpha\otimes e_\alpha\right). 
\end{align*}
Next we take $\sum_\alpha u_\infty^\alpha\otimes e_\alpha\in H_\infty\otimes V$ arbitrarily. 
By the assumption there are $u_i^\alpha\in H_i$ such that 
$u_i^\alpha\to u_\infty^\alpha$ strongly and 
$\mathcal{E}_\infty( u_\infty^\alpha)=
\lim_{i\to\infty}\mathcal{E}_i(u_i^\alpha)$. 
Then $\sum_\alpha u_i^\alpha\otimes e_\alpha \to 
\sum_\alpha u_\infty^\alpha\otimes e_\alpha$ 
strongly by Proposition \ref{tensor conv8.1} and 
we obtain 
\begin{align*}
\mathcal{E}_\infty^V\left(\sum_\alpha u_\infty^\alpha\otimes e_\alpha\right)
= \sum_\alpha \mathcal{E}_\infty
(u_\infty^\alpha)
&= \sum_\alpha \lim_{i\to\infty} \mathcal{E}_i
(u_i^\alpha)\\
&= \lim_{i\to\infty}\sum_\alpha \mathcal{E}_i
(u_i^\alpha)\\
&= \lim_{i\to\infty}\mathcal{E}_i^V\left(\sum_\alpha u_\infty^\alpha\otimes e_\alpha\right). 
\end{align*}
Finally, we take 
$\{ \sum_\alpha u_i^\alpha\otimes e_\alpha\}_i$ 
such that 
\begin{align*}
\limsup_{i\to \infty}\left( \left\| 
\sum_\alpha u_i^\alpha\otimes e_\alpha
\right\|_{H_i}^2+\mathcal{E}_i^V\left(\sum_\alpha u_i^\alpha\otimes e_\alpha\right)\right)
< \infty. 
\end{align*}
Since 
\begin{align*}
\left\| 
\sum_\alpha u_i^\alpha\otimes e_\alpha
\right\|_{H_i}^2
+\mathcal{E}_i^V\left(\sum_\alpha u_i^\alpha\otimes e_\alpha\right)
= \sum_\alpha \{ \| u_i^\alpha \|_{H_i}^2
+ \mathcal{E}_i( u_i^\alpha)\}, 
\end{align*}
we have $\limsup_{i\to\infty}\| u_i^\alpha \|_{H_i}^2
+ \mathcal{E}_i( u_i^\alpha) <\infty$ for 
all $\alpha$. 
Then we have 
a subsequence $\{ i_j\}_j\subset \{ i\}_i$ such that 
$\{ u_{i_j}^\alpha\}_j$ strongly converge 
for all $\alpha$ if $\mathcal{E}_i\cpt
\mathcal{E}_\infty$. 
Then $\{ \sum_\alpha u_{i_j}^\alpha\otimes e_\alpha\}_j$ 
strongly converges to some points by Proposition \ref{tensor conv8.1}. 
\end{proof}
As a consequence, we obtain the following. 
\begin{prop}
Let $\Sigma_i^V$ be the spectral structure 
on $H_i^V$ generated by $A_i^V$. 
If $\Sigma_i\str \Sigma_\infty$ 
(resp. $\Sigma_i\cpt \Sigma_\infty$), 
then $\Sigma_i^V\str \Sigma_\infty^V$ 
(resp. $\Sigma_i^V\cpt \Sigma_\infty^V$). 
\label{tensor conv2}
\end{prop}

\subsection{Closed subspaces}
Let $K_i\subset H_i$ be closed subspaces. 
In this subsection we assume the followings. 
\begin{itemize}
\setlength{\parskip}{0cm}
\setlength{\itemsep}{0cm}
 \item[$(i)$] For $\mathcal{C}':=\mathcal{C}\cap K_\infty$ and $\mathcal{C}'':=\mathcal{C}\cap K_\infty^\perp$, we have the decomposition 
$\mathcal{C}=\mathcal{C}'\oplus \mathcal{C}''$. 
 \item[$(ii)$] $\Phi_i(\mathcal{C}')\subset K_i$ 
and $\Phi_i(\mathcal{C}'')\subset K_i^\perp$. 
\end{itemize}
Moreover, denote by $\pi_{K_i}\colon H_i\to K_i$ 
and $\pi_{K_i^\perp}\colon H_i\to K_i^\perp$ 
the orthogonal projections. 

\begin{prop}
Let $u_i\in H_i$ and $u_i\to u_\infty$ strongly 
with respect to $\{ \Phi_i\}_i$. 
Then $\pi_{K_i}(u_i) \to \pi_{K_\infty}(u_\infty)$ 
strongly with respect to 
$\{ \Phi_i|_{\mathcal{C}'}\}_i$. 
\label{subspace conv1}
\end{prop}
\begin{proof}
Suppose $\{ \tilde{u}_k\}_k\in \mathcal{C}$ 
converges to $u_\infty$ and satisfies 
\begin{align*}
\lim_{k\to\infty}\limsup_{i\to\infty} 
\| \Phi_i(\tilde{u}_k) - u_i\|_{H_i} = 0.
\end{align*} 
Take $\{ \tilde{u}_k'\}_k\subset \mathcal{C}'$ 
and $\{ \tilde{u}_k''\}_k\subset \mathcal{C}''$ 
such that $\tilde{u}_k=\tilde{u}_k' + \tilde{u}_k''$. 
Then 
\begin{align*}
\| \Phi_i(\tilde{u}_k) - u_i\|_{H_i}^2 
&= \| \Phi_i(\tilde{u}_k') - \pi_{K_i}(u_i)\|_{H_i}^2
+ \| \Phi_i(\tilde{u}_k'') - \pi_{K_i^\perp}(u_i)\|_{H_i}^2 
\end{align*}
holds, which gives
\begin{align*}
\lim_{k\to\infty}\limsup_{i\to\infty} 
\| \Phi_i(\tilde{u}_k') - \pi_{K_i}(u_i)\|_{H_i} = 0.
\end{align*} 
Since $\tilde{u}_k' \to \pi_{K_\infty}(u_\infty)$, 
we have the assertion. 
\end{proof}

\begin{prop}
Let $u_i\in K_i$. 
\begin{itemize}
\setlength{\parskip}{0cm}
\setlength{\itemsep}{0cm}
 \item[$(1)$] $u_i \to 
u_\infty$ strongly with respect to 
$\{ \Phi_i|_{\mathcal{C}'}\}_i$
iff 
$u_i \to u_\infty$ strongly with respect to 
$\{ \Phi_i\}_i$. 
 \item[$(2)$] $u_i \to 
u_\infty$ weakly with respect to 
$\{ \Phi_i|_{\mathcal{C}'}\}_i$
iff 
$u_i \to u_\infty$ weakly with respect to 
$\{ \Phi_i\}_i$. 
\end{itemize}
\label{subspace conv2}
\end{prop}

\begin{proof}
Suppose $u_i \to u_\infty$ strongly with respect to 
$\{ \Phi_i\}_i$. 
Then there is $\tilde{u}_k\in \mathcal{C}$ converging to $u_\infty$ as $k\to\infty$ such that 
\begin{align*}
\lim_{k\to\infty}\limsup_{i\to\infty} 
\| \Phi_i(\tilde{u}_k) - u_i\|_{H_i} = 0.
\end{align*}
Since $\{ \pi_{K_\infty}\tilde{u}_k\}_k\subset 
\mathcal{C}'$ converges 
to $u_\infty$, we can show that 
$u_i \to u_\infty$ strongly with respect to 
$\{ \Phi_i|_{\mathcal{C}'}\}_i$. 
The converse is trivial. 

Next we show $(2)$. 
Suppose $u_i\to u_\infty$ weakly with respect to 
$\{ \Phi_i|_{\mathcal{C}'}\}_i$. 
Then it suffices to show that 
$\lim_{i\to\infty}\langle u_i,v_i\rangle$ 
for any strongly converging sequence 
$v_i\in H_i$. 
Since $\pi_{K_i}(v_i)\to \pi_{K_\infty}(v_\infty)$ strongly with respect to 
$\{ \Phi_i\}$ by the previous proposition, 
then $\pi_{K_i}(v_i)\to \pi_{K_\infty}(v_\infty)$ strongly with respect to 
$\{ \Phi_i|_{\mathcal{C}'}\}$ by $(1)$. 
Then we can deduce that 
\begin{align*}
\lim_{i\to\infty}\langle u_i,v_i\rangle
= \lim_{i\to\infty}\langle u_i,\pi_{K_i}(v_i)\rangle 
= \langle u_\infty,\pi_{K_\infty}(v_\infty)\rangle. 
\end{align*}
The converse is trivial.
\end{proof}

Assume that $A_i$ preserves the decomposition 
$H_i=K_i\oplus K_i^\perp$ in the following sense; 
\begin{itemize}
\setlength{\parskip}{0cm}
\setlength{\itemsep}{0cm}
 \item[$(iii)$] $\mathcal{D}(A_i)=\mathcal{D}(A_i)\cap K_i\oplus \mathcal{D}(A_i)\cap K_i^\perp$,
 \item[$(iv)$] $A_i(\mathcal{D}(A_i)\cap K_i)
\subset K_i$ and $A_i(\mathcal{D}(A_i)\cap K_i^\perp)
\subset K_i^\perp$. 
\end{itemize}

We can see that $\mathcal{E}_i|_{K_i}$ 
is closed quadratic form on $K_i$ 
since $\mathcal{E}_i$ is closed. 
Moreover, we have the followings by combining 
Propositions \ref{subspace conv1} and \ref{subspace conv2}. 
\begin{prop}
If $\mathcal{E}_i \Mosco \mathcal{E}_\infty$ 
(resp. $\mathcal{E}_i \cpt \mathcal{E}_\infty$), then $\mathcal{E}_i|_{K_i} \Mosco \mathcal{E}_\infty|_{K_\infty}$ 
(resp. $\mathcal{E}_i|_{K_i} \cpt \mathcal{E}_\infty|_{K_\infty}$). 
\end{prop}
\begin{prop}
Let $\Sigma_i|_{K_i}$ be the spectral structure 
on $K_i$ generated by 
$A_i|_{\mathcal{D}(A_i)\cap K_i}$. 
If $\Sigma_i\str \Sigma_\infty$ 
(resp. $\Sigma_i\cpt \Sigma_\infty$), then $\Sigma_i|_{K_i}
\str \Sigma_\infty|_{K_\infty}$ 
(resp. $\Sigma_i|_{K_i}
\cpt \Sigma_\infty|_{K_\infty}$). 
\label{subspace conv3}
\end{prop}

\subsection{$G$-invariant subspaces}
\label{G-inv}
Now let $G$ be a compact topological group 
and $(\rho,V)$ be a unitary representation of 
$G$. 
Suppose $G$ acts isometrically on each $H_i$, 
$G\cdot \mathcal{C}\subset \mathcal{C}$ 
and $\Phi_i$ are all $G$-equivariant. 
Then $(\mathcal{C}\otimes V)^G$ is dense in 
$(H_\infty\otimes V)^G$ and 
$\Phi_i$ and $\Sigma_i$ induce 
$\Phi_i^\rho:=\Phi_i^V|_{(\mathcal{C}\otimes V)^G}\colon 
(\mathcal{C}\otimes V)^G \to (H_i\otimes V)^G$ 
and $\Sigma_i^\rho:=\Sigma_i^V|_{(H_i\otimes V)^G}$. 
By Propositions \ref{tensor conv2} 
and \ref{subspace conv3}, we have the 
following result. 
\begin{prop}
If $\Sigma_i\str \Sigma_\infty$ 
(resp. $\Sigma_i\cpt \Sigma_\infty$), then 
$\Sigma_i^\rho \str \Sigma_\infty^\rho$ 
(resp. $\Sigma_i^\rho \cpt \Sigma_\infty^\rho$). 
\label{G-subspace conv}
\end{prop}

\section{Main result}\label{sec main}
Combining Propositions \ref{replace of Phi} 
and \ref{G-subspace conv} 
we have the following result. 
\begin{thm}
Let $G$ be a compact Lie group, 
\begin{align*}
&(P_i,d_i,\nu_i) \in \mathcal{M}(n,\kappa,D) 
,\quad (P_\infty,d_\infty,\nu_\infty)\in 
\overline{\mathcal{M}(n,\kappa,D)}\\
({\rm resp.}\ &(P_i,d_i,\nu_i,p_i) \in \mathcal{M}(n,\kappa), \quad
(P_\infty,d_\infty,\nu_\infty,p_\infty) \in \overline{\mathcal{M}(n,\kappa)} )
\end{align*}
have isometric $G$-actions and 
\begin{align*}
(P_i,d_i,\nu_i) &\GmGH 
(P_\infty,d_\infty,\nu_\infty).\\
({\rm resp.}\, (P_i,d_i,\nu_i,p_i) &\GmGH (P_\infty,d_\infty,\nu_\infty).)
\end{align*}
Let $\Sigma_i$ be the spectral structures 
generated by $\Delta_{\nu_i}$ 
and $(\rho,V)$ be a finite dimensional unitary representation of $G$. 
Then $\Sigma_i^\rho \cpt \Sigma_\infty^\rho$ 
(resp. $\Sigma_i^\rho \str \Sigma_\infty^\rho$) as $i\to \infty$. 
\label{main 9-1}
\end{thm}

Next we show the convergence of the 
eigenvalues and eigenfunctions. 
For $(P,d,\nu) \in 
\overline{\mathcal{M}(n,\kappa,D)}$ 
with an isometric $G$-action, 
$(L^2(P,\nu)\otimes V)^G$ may be 
of finite dimension. 
In this case, the $j$-th eigenvalue of 
the Laplacian acting on 
$(L^2(P,\nu)\otimes V)^G$ 
is defined to be infinity if 
\begin{align*}
j>\dim (L^2(P,\nu)\otimes V)^G.
\end{align*}
\begin{thm}
Let $G$ be a compact Lie group, 
$(P_i,d_i,\nu_i) \in 
\overline{\mathcal{M}(n,\kappa,D)}$ 
have isometric $G$-actions for 
all $i\in\Z_{\ge 0}\cup\{\infty\}$ and 
\begin{align*}
(P_i,d_i,\nu_i) \GmGH 
(P_\infty,d_\infty,\nu_\infty).
\end{align*}
Denote by $\lambda_{i,j}^{\rho}$ the $j$-th 
eigenvalue of 
$\Delta_{\nu_i}\colon (\mathcal{D}(\Delta_{\nu_i})\otimes V)^G 
\to (L^2(P_i)\otimes V)^G$ with maultiplicity 
and let $\{ f_{\infty,j}\}_{j=1}^N$ be the complete 
orthonormal system of 
$(L^2(P_\infty)\otimes V)^G$, where 
$N=\dim (L^2(P_\infty)\otimes V)^G$, such that 
$\Delta_{\nu_\infty}f_{\infty,j}
=\lambda_{\infty,j}^{\rho}f_{\infty,j}$. 
Then we have 
\begin{align*}
\lim_{i\to\infty} \lambda_{i,j}^{\rho} = \lambda_{\infty,j}^{\rho}
\end{align*}
for all $j=\in\Z_{>0}$, and there are orthonormal systems 
$\{ f_{i,j}\}_{j=1}^N$ of $(L^2(P_i)\otimes V)^G$ 
such that 
$\Delta_{\nu_i}f_{i,j}=\lambda_{i,j}^{\rho}f_{i,j}$ 
and 
$\lim_{i\to\infty}\| f_{i,j}-f_{\infty,j}\circ\phi_i\|_{L^2} = 0$, 
where $\phi_i$ is defined as in Definition \ref{def mGH}. 
\label{main 9-2}
\end{thm}
\begin{proof}
Let $\hat{\Phi}_i$ be as in 
Proposition \ref{replace of Phi} 
and $\hat{\Phi}_i^\rho$ be as in Section \ref{G-inv}. 
By \cite[Theorem 6.27]{Cheeger-Colding3}, 
every eigenfunctions $f_{\infty,j}$ is in $\mathcal{C}=(C(P_\infty)\otimes V)^G$. 
Then by applying \cite[Corollary 2.5]{KuwaeShioya2003} 
to our case, 
we have the assertion. 
\end{proof}

Let $P$ be a principal $G$-bundle, 
$A$ be a $G$-connection on $P$, 
$(\rho,V)$ be a real unitary irreducible 
representation of $G$. 
Put $E_\rho=P\times_\rho V$ and 
let $\nabla=\nabla^A$ is induced from 
$A$, and $\lambda_j^\nabla$ 
be the $j$-th eigenvalue of the rough Laplacian 
$\nabla^*\nabla$. 
As a consequence of the above theorem, 
we have the following result. 
\begin{thm}
Let $G$ be a compact Lie group. 
For any $\kappa\in\R,D,N>0$, 
there exist constants $0\le C_j$ 
depending only on 
$j,n,\kappa,D,N,G,\rho,V$ such that 
$\lim_{j\to\infty}C_j=\infty$ and 
the following holds. 
For any closed Riemannian manifold 
$(X,g)$  
of dimension $n$ and principal 
$G$-bundle $\pi\colon P\to X$ with 
the $G$-connection $A$ such that 
\begin{align*}
&{\rm dim}\, X = n,\quad 
{\rm Ric}_g \ge \kappa g,\quad 
{\rm diam}\, X \le D,\\
&\| (d^{\nabla^A})^* F^A \|_{L^\infty} \le N, \quad
\| F^A \|_{L^\infty} \le N,
\end{align*}
we have 
\begin{align*}
\lambda_j^{\nabla^{A}} \ge C_j .
\end{align*}
\label{bdd of eigenvalues}
\end{thm}
\begin{proof}
For fixed $n,\kappa,D,N$, 
let $\mathcal{M}_G(n,\kappa,D,N,\sigma)$ 
be the space consists of isometric classes of 
the Riemannian manifold $(P,h(g,A,\sigma))$ 
satisfying the assumptions. 
By Theorem \ref{precompact}, 
$\mathcal{M}_G(n,\kappa,D,N,\sigma)$ is precompact with respect to 
the asymptotically $G$-equivariant measured Gromov-Hausdorff topology. 
The function which maps 
$(P,h(g,A,\sigma))$ 
to $\lambda_j^\nabla=\lambda_j^\rho-\chi_{\sigma,\rho}$ is continuous 
by Section \ref{principal metric} and Theorem \ref{main 9-2}. 
Now we deny the conclusion. 
Suppose there are a constant $B>0$, 
an increasing sequence $1\le j_1<j_2<j_3<\cdots$, 
principal bundles 
$P_k\to X_k$ with $G$-connections $A_k$ 
such that $\lambda_{j_k}^{\nabla^{A_k}} \le B$ 
for every $k$. 
Then by the precompactness of 
$\mathcal{M}_G(n,\kappa,D,N,\sigma)$, 
there exists the limit $(P,d,\nu)$ of 
some subsequences of $\{ P_k\}_k$. 
On $(P,d,\nu)$, one can check that $\lambda_{j}^\rho\le B+\chi_{\rho,\sigma}$ for all $j$, 
hence we have the contradiction. 
\end{proof}

To show the uniform upper bound 
of the eigenvalues, we need to add some assumptions 
since the eigenvalues may be $\infty$ 
if $(L^2(P)\otimes V)^G$ is of finite dimension. 

\begin{lem}
Let $G$ be a compact topological group and 
$(P,d)$ be a proper metric space 
with an isometric $G$-action. 
Then for any 
orthogonal representation $(\rho,V)$ of $G$, 
\begin{align*}
\dim (C_0(P)\otimes V)^G
\ge \sharp\{ \bar{u}\in P/G;\, V^{G_u}\neq\{ 0\}\}
\end{align*}
holds, where $G_u=\{ \gamma\in G;\, u\gamma =u\}$. 
\label{lem infinite dim}
\end{lem}
\begin{proof}
Take any finite subset 
\begin{align*}
\{ \bar{u}_1,\bar{u}_2, \ldots,\bar{u}_N\}\subset \{ \bar{u}\in P/G;\, V^{G_u}\neq\{ 0\}\}.
\end{align*}
Then we can take $\delta>0$ such that 
$B(\bar{u}_i,2\delta)\cap B(\bar{u}_j,2\delta)=\emptyset$ 
for any $i\neq j$. 
Define $f_i\in C_0(P)$ by 
\begin{align*}
f_i(u):=\frac{d\left( u,B(u_i,2\delta)^c\right)}{d\left( u,\overline{B(u_i,\delta)}\right) + d\left( u,B(u_i,2\delta)^c\right)},
\end{align*}
then $f_i$ is nonnegative on $P$, 
$f_i|_{\overline{B(u_i,\delta)}}\equiv 1$ and 
${\rm supp}(f_i)\subset B(u_i,2\delta)$ holds. 
Let $\mu_G$ be the Haar 
measure on $G$.
Take $v_i\in V^{G_{u_i}}$ such that $\| v_i\|=1$ and 
put 
\begin{align*}
\hat{f}_i(u):=\int_G f_i(u\gamma) \rho(\gamma)v_i d\mu_G (\gamma)\in V, 
\end{align*}
then one can check that $\hat{f}_i\in (C_0(P)\otimes V)^G$ and 
${\rm supp}(\hat{f}_i)\subset \pi^{-1}(B(\bar{u}_i,2\delta))$, 
where $\pi\colon P\to P/G$ is the quotient map. 
Next we show the linear independence of 
$\{ \hat{f}_i\}_{i=1}^N$. 
Since ${\rm supp}(\hat{f}_i)\cap 
{\rm supp}(\hat{f}_j)=\emptyset$ if $i\neq  j$, 
it suffices to show that $\hat{f}_i(u_i)\neq 0$. 

Define the continuous maps 
$\varphi_P\colon G_{u_i}\backslash G\to P$ and 
$\varphi_V\colon G_{u_i}\backslash G\to \R$ by 
$\varphi_P(G_{u_i}\cdot \gamma)= u_i\gamma$ and 
$\varphi_V(G_{u_i}\cdot \gamma)= 
\langle \rho(\gamma)v_i,v_i\rangle_V$, respectively. 
Since $\varphi_P$ is a homeomorphism, 
then $\varphi_V\circ\varphi_P^{-1}$ is continuous 
therefore 
we can replace $\delta$ by the smaller one 
such that 
\begin{align*}
\left| 
\langle \rho(\gamma)v_i, v_i\rangle_V -\| v_i\|_V^2
\right| < \frac{1}{10}
\end{align*}
holds for 
any $u_i\gamma\in B(u_i,2\delta)$. 
Then we can deduce that $f_i(u_i\gamma)\neq 0$ implies 
$\left| 
\langle \rho(\gamma)v_i, v_i\rangle_V -\| v_i\|_V^2
\right|<\frac{1}{10}$, therefore, 
\begin{align*}
\left| f_i(u_i \gamma)\langle \rho(\gamma)v_i, v_i\rangle_V  - f_i(u_i \gamma) \| v_i\|_V^2\right|
\le \frac{ f_i(u_i \gamma) }{10}
\end{align*}
and 
\begin{align*}
\left| \langle \hat{f}_i(u_i),v_i\rangle_V 
- \int_G f_i(u_i\gamma) \| v_i\|_V^2 
d\mu_G (\gamma)\right|
\le \frac{1}{10}\int_G f_i(u_i\gamma) 
d\mu_G (\gamma)
\end{align*}
holds. 
Since 
\begin{align*}
\int_G f_i(u_i\gamma) \| v_i\|_V^2 d\mu_G (\gamma)
= \int_G f_i(u_i\gamma) d\mu_G (\gamma),
\end{align*}
we have 
\begin{align*}
\langle \hat{f}_i(u_i),v_i\rangle_V 
\ge \frac{9}{10}\int_G f_i(u_i\gamma) d\mu_G (\gamma).
\end{align*}
Define $\psi\colon G\to P$ by 
$\psi(\gamma)=u_i\gamma$. Then 
\begin{align*}
\int_G f_i(u_i\gamma) d\mu_G (\gamma)
\ge 
\mu_G (\psi^{-1}(B(u_i,\delta)) )>0
\end{align*}
holds since $\psi^{-1}(B(u_i,\delta))$ is 
nonempty and open, 
hence $\langle \hat{f}_i(u_i),v_i\rangle_V \neq 0$ 
holds. 
\end{proof}

\begin{rem}
\normalfont
In fact, we can show that 
\begin{align*}
\dim (C_0(P)\otimes V)^G
= \sum_{\bar{u}\in P/G}\dim V^{G_u}\in
\N\cup \{ \infty\}
\end{align*}
by more precise argument. 
\end{rem}

For a compact Lie group $G$ and 
an orthogonal representation $(\rho,V)$, 
put 
\begin{align*}
S(G)&:=\{ H\subset G;\, H\mbox{ is a closed subgroup}\},\\
S(G,V)&:=\{ H\in S(G);\, V^H=\{ 0\}\},
\end{align*}
and for a metric space $(P,d)$ with 
an isometric $G$-action, put 
\begin{align*}
\delta_V(u)
&:= \inf_{H\in S(G,V)} 
\sup_{h\in H}d(u,uh).
\end{align*}
Now we have the following uniform 
upper bound of eigenvalues.
\begin{thm}
Let $G$ be a compact Lie group. 
For any $\kappa\in\R,D,N>0$, 
there exist constants $0\le C_j\le C'_j$ 
depending only on 
$j,n,\kappa,D,N,G,\rho,V$ such that 
$\lim_{j\to\infty}C_j=\infty$ and 
the following holds. 
For any closed Riemannian manifold 
$(X,g)$  
of dimension $n$ and principal 
$G$-bundle $\pi\colon P\to X$ with 
the $G$-connection $A$ such that 
\begin{align*}
&{\rm dim}\, X = n,\quad 
{\rm Ric}_g \ge \kappa g,\quad 
D^{-1} \le {\rm diam}\, X \le D, \\
&\| (d^{\nabla^A})^* F^A \|_{L^\infty} \le N, \quad
\| F^A \|_{L^\infty} \le N, \quad 
\sup_{P}\delta_V \ge N^{-1}
\end{align*}
we have 
\begin{align*}
C_j \le \lambda_j^{\nabla^{A}} \le C'_j.
\end{align*}
\label{bdd of eigenvalues 2}
\end{thm}
\begin{proof}
If we show that the 
$(C_0(P_\infty)\otimes V)^G$ 
is of infinite dimensional 
for any $(P_\infty,d_\infty)$ which 
is the asymptotically $G$-equivariant 
Gromov-Hausdorff limit of 
principal $G$-bundles satisfying 
the assumptions, then 
we obtain the upper estimate 
similarly as the proof of 
Theorem \ref{bdd of eigenvalues}. 

By the triangle inequality 
we have 
\begin{align*}
|d_\infty(u,uh)-d_\infty(u',u'h)|\le 2d_\infty(u,u')
\end{align*}
for any $u,u'\in P_\infty$ and $h\in G$, 
hence the function $\delta_V\colon 
P_\infty\to \R$ is continuous. 
Moreover we can check that 
$(P,d)\mapsto \sup_P \delta_V$ 
is continuous with respect to asymptotically 
$G$-equivariant Gromov-Hausdorff topology, 
hence we have 
$\sup_{P_\infty}\delta_V>0$ and 
${\rm diam}\, P_\infty/G>0$.
Since $P_\infty$ is the length space, 
there exists infinitely many $x\in P_\infty/G$ 
such that $\delta_V(u)>0$ for 
any $u\in P_\infty$ with $\bar{u}=x$. 
Since 
$\delta_V(u)>0$ implies 
$G_u\notin S(G,V)$, 
we can see that 
$\{ \bar{u}\in P_\infty/G;\, V^{G_u}\neq
\{ 0\}\}$ is an infinite set, 
hence $\dim (C_0(P_\infty)\otimes V)^G$ 
is infinity 
by Lemma \ref{lem infinite dim}. 
\end{proof}

Finally we describe the spectral 
convergence in the case of noncompact. 

\begin{thm}
Let $G$ be a compact Lie group, 
$(P_i,d_i,\nu_i,p_i) \in 
\overline{\mathcal{M}(n,\kappa)}$ 
have isometric $G$-actions for 
all $i\in\Z_{\ge 0}\cup\{\infty\}$ and 
\begin{align*}
(P_i,d_i,\nu_i,p_i) \GmGH 
(P_\infty,d_\infty,\nu_\infty,p_\infty).
\end{align*}
Denote by $\sigma(\Delta_{\nu_i}) \subset \R$ the 
set consists of all eigenvalues of 
$\Delta_{\nu_i}\colon (\mathcal{D}(\Delta_{\nu_i})\otimes V)^G 
\to (L^2(P_i)\otimes V)^G$. 
Then for any $\lambda\in\sigma(\Delta_{\nu_\infty}) 
$ there are $\lambda_i\in\sigma(\Delta_{\nu_i})$ 
such that $\lim_{i\to\infty}\lambda_i =\lambda$. 
\label{main 9-4}
\end{thm}
\begin{proof}
It follows from 
Theorem \ref{main 9-1} and 
\cite[Proposition 2.5]{KuwaeShioya2003}. 
\end{proof}

\section{Applications to holomorphic 
vector bundles on complex manifolds}
\label{holomorphic}
In this section we show Theorem 
\ref{main 1-2}. 
Let $(X,\omega)$ be a compact K\"ahler manifold 
and $\pi_E\colon E\to X$ be a holomorphic vector bundle 
over $X$. 
Suppose $h$ is a hermitian metric on $E$ and 
$\nabla\colon \Gamma(E)\to\Omega^1(E)$ is the Chern connection. 
Under the decomposition $\Omega^1=\Omega^{1,0}\oplus \Omega^{0,1}$, 
we have the decomposition 
$\nabla=\nabla^{1,0}+\nabla^{0,1}$. 
Let $\nabla^*,(\nabla^{1,0})^*,(\nabla^{0,1})^*$ are the formal adjoint 
of $\nabla,\nabla^{1,0},\nabla^{0,1}$, respectively. 
For a holomorphic coordinate $(U,z^1,\ldots,z^n)$ on $X$ 
we put $\omega=\sqrt{-1}g_{i\bar{j}}dz^i\wedge d\bar{z}^j$. 
Then we may write 
\begin{align*}
\nabla^* &= (\nabla^{1,0})^* + (\nabla^{0,1})^*,\\
(\nabla^{1,0})^* &= -g^{i\bar{j}}\iota_{\partial_i}\nabla_{\bar{\partial}_j},\\
(\nabla^{0,1})^* &= -g^{i\bar{j}}\iota_{\bar{\partial}_j}\nabla_{\partial_i},
\end{align*}
where $\partial_i:=\frac{\partial}{\partial z^i}$. 
Let $F^\nabla\in\Omega^{1,1}({\rm End}E)$ 
be the curvature form. 
Since we have 
\begin{align*}
(\nabla^{1,0})^*\nabla^{1,0}s
&= (\nabla^{0,1})^*\nabla^{0,1}s
+ g^{i\bar{j}} F^\nabla(\partial_i, \bar{\partial}_j)s,
\end{align*}
we obtain 
\begin{align*}
\nabla^*\nabla &= 2\Delta_{\bar{\partial}} + \Lambda_\omega F^\nabla,\\
\Delta_{\bar{\partial}} &:= (\nabla^{0,1})^*\nabla^{0,1},\\
\Lambda_\omega F^\nabla &:= g^{i\bar{j}} F(\partial_i, \bar{\partial}_j)\in\Gamma({\rm End}E). 
\end{align*}
A {\it Hermitian-Einstein connection} 
is the Chern connection $\nabla$ 
which satisfies 
\begin{align*}
\Lambda_\omega F^\nabla=\mu\cdot {\rm id}_E 
\end{align*}
for some constant $\mu\in\R$. 
In this case $\mu$ is given by 
\begin{align*}
\mu= \frac{2\pi \dim X \cdot 
c_1(E)\cdot [\omega]^{n-1}}
{{\rm rk}(E)[\omega]^n}.
\end{align*}

Hence a smooth section of $E$ is 
holomorphic if and only if 
it is the eigensection of $\nabla^*\nabla$ 
corresponding to $\mu$. 

Here, $\mu$ is determined by 
the topological invariant 
depending only on the K\"ahler class $[\omega]$  
and the first Chern class $c_1(E)$. 
It is known that the Hermitian-Einstein connections 
are Yang-Mills, hence 
$(d^{\nabla})^*F^\nabla=0$.

Now we apply Theorem \ref{bdd of eigenvalues} 
to this case. 
To apply it, we fix the dimension 
of the manifolds and, the rank of the 
holomorphic vector bundles and 
the constant $\mu$, 
assume that the connections are 
Hermitian-Einstein and 
give the uniform bound of 
the diameters, the Ricci curvatures and 
the $L^\infty$-norm of the curvature 
forms of the connections. 
Then we have the uniform lower bound 
$C_j$ of the $j$-th eigenvalue of the 
connection Laplacian such that 
$\lim_{j\to \infty} C_j=\infty$. 
If we take $j$ such that 
$C_j>\mu$, then we can see that 
$j$ is the upper bound of 
the dimension of 
\begin{align*}
H^0(X,E)
=\{ s\colon X\to E;\, s\mbox{ is a holomorphic section}\}.
\end{align*} 
Thus we have shown Theorem \ref{main 1-2}.

Moreover we suppose that $L=E$ is a 
holomorphic line bundle such that 
$c_1(L)=[\frac{\omega}{2\pi}]$. 
Then there is a hermitian metric such that 
$F^\nabla=-\sqrt{-1}\omega$, 
hence $\| F^\nabla\|_{L^\infty}$ is the universal constant. 
Therefore, we obtain the next result 
as the special case of Theorem 
\ref{main 1-2}.

\begin{thm}
For any $\kappa\in\R,D>0$ 
and $n\in\Z_{>0}$, 
there exist a positive integer $N$ 
depending only on 
$\kappa,D,n$ such that 
the following holds. 
For any closed K\"ahler manifold 
$(X,\omega)$ and holomorphic line 
bundle
$L\to X$ 
such that 
\begin{align*}
&\dim X = n,\quad 
c_1(L)=\left[ \frac{\omega}{2\pi} \right],\quad 
{\rm Ric}_\omega \ge \kappa \omega,\quad 
{\rm diam}_\omega X \le D,
\end{align*}
we have 
\begin{align*}
\dim H^0(X,L) \le N.
\end{align*}
\label{bdd of holomorphic sections 2}
\end{thm}

\bibliographystyle{plain}

\begin{thebibliography}{10}

\bibitem{Cheeger-Colding1997}
Jeff Cheeger and Tobias~H. Colding.
\newblock On the structure of spaces with {R}icci curvature bounded below. {I}.
\newblock {\em J. Differential Geom.}, 46(3):406--480, 1997.

\bibitem{Cheeger-Colding3}
Jeff Cheeger and Tobias~H. Colding.
\newblock On the structure of spaces with {R}icci curvature bounded below.
  {III}.
\newblock {\em J. Differential Geom.}, 54(1):37--74, 2000.

\bibitem{Fukaya1987}
Kenji Fukaya.
\newblock Collapsing of {R}iemannian manifolds and eigenvalues of {L}aplace
  operator.
\newblock {\em Invent. Math.}, 87(3):517--547, 1987.

\bibitem{FuakayaYamaguchi1994}
Kenji Fukaya and Takao Yamaguchi.
\newblock Isometry groups of singular spaces.
\newblock {\em Math. Z.}, 216(1):31--44, 1994.

\bibitem{Honda2017spectral}
Shouhei Honda.
\newblock Spectral convergence under bounded {R}icci curvature.
\newblock {\em J. Funct. Anal.}, 273(5):1577--1662, 2017.

\bibitem{Kasue2002}
Atsushi Kasue.
\newblock Convergence of {R}iemannian manifolds and {L}aplace operators. {I}.
\newblock {\em Ann. Inst. Fourier (Grenoble)}, 52(4):1219--1257, 2002.

\bibitem{Kasue2011}
Atsushi Kasue.
\newblock Spectral convergence of {R}iemannian vector bundles.
\newblock {\em Sci. Rep. Kanazawa Univ.}, 55:25--49, 2011.

\bibitem{KuwaeShioya2003}
Kazuhiro Kuwae and Takashi Shioya.
\newblock Convergence of spectral structures: a functional analytic theory and
  its applications to spectral geometry.
\newblock {\em Comm. Anal. Geom.}, 11(4):599--673, 2003.

\bibitem{Lott-Dirac2002}
John Lott.
\newblock Collapsing and {D}irac-type operators.
\newblock In {\em Proceedings of the {E}uroconference on {P}artial
  {D}ifferential {E}quations and their {A}pplications to {G}eometry and
  {P}hysics ({C}astelvecchio {P}ascoli, 2000)}, volume~91, pages 175--196,
  2002.

\bibitem{Lott-form2002}
John Lott.
\newblock Collapsing and the differential form {L}aplacian: the case of a
  smooth limit space.
\newblock {\em Duke Math. J.}, 114(2):267--306, 2002.

\end{thebibliography}

\end{document}